\documentclass[12pt,notitlepage]{report}
\title{Non extendability from any side of the domain of definition as a 
generic property of smooth or simply continuous functions on an analytic curve}
\author{E. Bolkas, V. Nestoridis and C. Panagiotis}
\date{}
\usepackage[english]{babel}
\usepackage[utf8]{inputenc}
\usepackage{titlesec}
\usepackage{amsthm}
\usepackage{amsmath}
\usepackage{amsfonts}
\usepackage{mathrsfs}
\usepackage{enumerate}

\titleformat*{\section}{\normalsize\bfseries}
\titleformat*{\subsection}{\Large\bfseries}

\titleformat*{\section}{\normalsize\bfseries}
\titleformat*{\subsection}{\Large\bfseries}
\newtheorem{theorem}{Theorem}[chapter]
\numberwithin{theorem}{section}
\newtheorem{defi}[theorem]{Definition}
\newtheorem{lemma}[theorem]{Lemma}
 
\newtheorem{prop}[theorem]{Proposition}

\theoremstyle{definition}

\newtheorem{remark}[theorem]{Remark}

\newcommand{\bbR}{\mathbb{R}}
\newcommand{\bbZ}{\mathbb{Z}}

\begin{document}
\maketitle
\begin{abstract}
\noindent 
In this article we show that extendability from one side of a simple analytic curve
is a rare phenomenon in the topological sense in various spaces of functions. Our
result can be proven using Fourier methods combined with other facts or by complex
analytic methods and a comparison of the two methods is possible.
\end{abstract}

\noindent AMS Classification $n^o$: 42A16, 30H10\\


\section{Introduction}
In [2] it is considered the space $C^\infty([0,1])$ endowed with its natural Frechet
topology. It is proven that the set $A$ of functions $f\in C^\infty([0,1])$
nowhere analytic contains a dense and $G_\delta$ subset of $C^\infty([0,1])$.
In the present paper  we strengthen the above result by showing that the set $A$ is 
it self $G_\delta$. Furthermore, we prove a similar result in the space of
$C^\infty$ periodic functions, that is in $C^\infty(T)$, where $T$ is the unit 
circle.
In order to do this we combine the fact that non extendability is a generic 
phenomenon for the space $A^\infty(D)$, where $D$ is the unit disk ([8]), with some 
elements of Fourier Analysis. More precisely, using the Cauchy transform, 
$C^\infty(T)$ is a direct sum of $A^\infty(D)$ and another space $Y$. This space
$Y$ contains functions holomorphic outside the open unit disk, which extend smoothly
on $T$. A careful observation of the above facts enables us to prove that generically
all functions in $C^\infty(T)$ are nowhere extendable from neither side of the 
unit circle $T$. We notice that extendability from one only side of the domain of
definition has already been considered in [4] and the references therein. There one
investigates sufficient conditions so that, if all derivatives of a function at a 
point of the boundary vanish one can conclude that the one side extension is 
identically equal to zero. It is, as far as we know, the first time where the 
phenomenon of one-side extendability is noticed to be a rare phenomenon.

Next, we tried to extend the previous results from the whole circle $T$ to a subarc 
of it (or equivalently to a segment). To do this we combined the previous result
with a theorem of Borel ([7]) saying that any complex sequence appears as the 
sequence
of derivatives at $0$ of a $C^\infty$ function and with a simple version of 
Michael's selection theorem [11]. The latter allows us to assign in a continuous way
a $C^\infty$ function $f$ to every complex sequence $a_n$, $n=0,1,2,...$ so that
$f^{(n)}(0)=a_n$ for all n. It was also necessary to localize a necessary and 
sufficient condition for non extendability of holomorphic functions in the open 
disk $D$. More precisely, it is known [8] that a holomorphic function $f\in H(D)$ is
nowhere extendable if and only if 
$R_\zeta=dist(\zeta,\partial D=T)$ for all $\zeta\in D$, where
$R_\zeta$ demotes the radius of convergence of the Taylor expansion of f with center 
$\zeta$.
The localized version of it is the following: A holomorphic function $f\in H(D)$
is not extendable at any open set containing $1\in T=\partial D$ if and only if 
$R_\zeta\leq |\zeta-1|$ for all $\zeta\in D$ (or equivalently for a denumerable 
subset of $\zeta$'s in D clustering to 1). 
In this way, the Fourier method, combined with other facts, yields the results for
any circular arc, segment, the circle T, the real line $\mathbb{R}$ for the space
$C^\infty$ and $C^k. k=0,1,2,...$, where $C^0$ denotes the space of continuous
functions and on some variations of these spaces. Furthermore, composing with
a conformal map, we can transfer the above result at any injective analytic curve
(replacing T, $\mathbb{R}$ or [0,1]). We also prove that, if a domain $\Omega$ in
$\mathbb{C}$ is bounded by a finite number of disjoint analytic Jordan curves, then
generically in $A^\infty(\Omega)$ every function f is nowhere real analytic on the
boundary of $\Omega$. When we say real analytic we mean with respect to a conformal
parameter of the boundary analytic curves $\gamma_i$, $i=1,2,...,n$.
But then naturally the following question arises. What happens if we parametrize 
$\gamma_i$ in a different way, for instance, by the arc length parameter? In order to
answer this it has been proven in [9] that the arc length parameter is a global 
conformal parameter for any analytic curve, thus, nothing changes if we use the arc 
length parameter.

At the last section of the present paper we present a complex analytic approach, 
based
on the Hardy space $H^\infty(D)$ and Montel's theorem, which yields all previous 
results in a much simpler and natural way. Thus, although in general the complex 
method works in less cases than the real-harmonic analysis method, but with simpler
and less combinatorial arguments, in the present case the complex method initially
was going much further than the Fourier method and with a much more natural and easy
way. In order to push the Fourier method to give almost as many results as the 
complex
method, we are led to combine it with Michael's selection theorem, with Borel's 
theorem and to localize non extendability results from [8]. The procedure is not 
simple but clearly interesting and beautiful.

\section{Preliminaries}
\begin{prop}
Let $L\subset \mathbb{C}$ be a closed set and $\Omega\subset \mathbb{C}$ be an open
set. Let $f: \Omega\rightarrow \mathbb{C}$ be a continuous function on $\Omega$, 
which is holomorphic at $\Omega \setminus L$. Then, if $L$ satisfies one of the 
following conditions i), ii), iii), iv) and v), automatically it follows that $f$ is
holomorphic at whole $\Omega$.
\begin{enumerate}[(i)]
\item $L=\left\{x+ig(x): x\in [a,b]\right\}$, where $g:[a,b]\rightarrow \mathbb{R}$
is continuous and of bounded variation (that is $L$ has finite length) and $a<b$.
\item $L$ is the boundary of a convex body in $\mathbb{C}$
\item $L$ is a straight line or circle or segment or circular arc.
\item $L$ is an analytic Jordan curve or a compact simple analytic arc.
\item $L$ is a conformal image of a set of the previous types.
\end{enumerate}
\end{prop}

\begin{proof}
i) Let $z_0 \in L$. We denote the open square of center $z_0$ and side $r>0$ as  
$E(z_0,r)$. We can find $r_0 > 0$ such that 
$\partial E(z_0,r_0))\setminus L$ has two connected components $E_1$ and
$E_2$. We denote by $\gamma$ the boundary $\partial E(z_0,r_0)$ of $E(z_0,r_0)$,
and define the function
$$F(z)=\dfrac{1}{2\pi i} \int_{\gamma}{\dfrac{f(w)}{w-z}dw},$$ which is holomorphic
at $\mathbb{C}\setminus \gamma$. We will prove that $F(z)=f(z)$ for $z\in E(z_0,r)$
and thus $f$ will be holomorphic at $E(z_0,r)$. Since $z_0$ was arbitrary at $L$ we 
will have proved that $f$ is holomorphic at $\Omega$. 

In order to prove that $F(z)=f(z)$ for $z\in E(z_0,r)$ we first notice that
$$F(z)=\dfrac{1}{2\pi i} \int_{\gamma_1}{\dfrac{f(w)}{w-z}dw}+\dfrac{1}{2\pi i} 
\int_{\gamma_2}{\dfrac{f(w)}{w-z}dw},$$ where $\gamma_1=(\overline{E(z_0,r)}\cap L) + 
E_1$, $\gamma_2=(\overline{E(z_0,r)}\cap L) + E_2$ and $\gamma_1, \gamma_2$
are counterclockwise oriented. 

Let $z\in E(z_0,r)\setminus L$ and $\varepsilon>0$. 
If z belongs to the interior of 
$\gamma_1$, then we will prove that
$$\dfrac{1}{2\pi i}\int_{\gamma_2}{\dfrac{f(w)}{w-z}dw}=0.$$
We can find a $t_0 \in \mathbb{R}$ such that
$f$ is holomorphic at the interior of
$\gamma_{2,t_0}=\gamma_2+it_0$, where $"+"$ means translation in $\mathbb{C}$, and 
the curves $\gamma_{2,\lambda t_0}=\gamma_2+i \lambda t_0$ do not intersect
$\gamma_2$ for $0<\lambda \leq 1$. For this it is sufficient to choose $t_0 >0$ or
$t_0<0$.
From the uniform continuity of $\dfrac{f(w)}{w-z}$ at any compact set uniformly
away from z we can find a $0<\lambda_0 \leq 1$ such that
$$\left|\dfrac{f(w)}{w-z}-\dfrac{f(w+ \lambda t_0)}{w+ \lambda t_0-z}\right|\leq
\dfrac{\pi \varepsilon}{M}$$ for $0 \leq \lambda \leq \lambda_0$ and $w\in \gamma_2$, 
where $M<\infty$ is the length of $\gamma_2$. Then
\begin{align*}
\left|\dfrac{1}{2\pi i} \int_{\gamma_2}{\dfrac{f(w)}{w-z}dw}-\dfrac{1}{2\pi i} 
\int_{\gamma_{2,\lambda t_0}}{\dfrac{f(w)}{w-z}dw}\right|=\dfrac{1}{2\pi}\left|
\int_{\gamma_2}{\Big(\dfrac{f(w)}{w-z}-{\dfrac{f(w+\lambda t_0)}
{w+\lambda t_0-z}\Big)dw}}\right|\leq \\  \leq \dfrac{1}{2\pi}\int_{\gamma_2}
{\left|\dfrac{f(w)}{w-z}-\dfrac{f(w+\lambda t_0)}{w+\lambda t_0-z}\right|dw}
 \leq \dfrac{1}{2\pi} \int_{\gamma_2}\dfrac{\pi\varepsilon}{M}dw
 =\dfrac{1}{2\pi}\dfrac{\pi\varepsilon}{M}M=\dfrac{\varepsilon}{2}<\varepsilon
\end{align*}
for every $0<\lambda \leq \lambda_0$.
Thus $$\lim_{\substack{\lambda\to 0  \\ \lambda>0}}\dfrac{1}{2\pi i}
\int_{\gamma_2,\lambda t_0}{\dfrac{f(w)}{w-z}dw}=
\dfrac{1}{2\pi i}\int_{\gamma_2}{\dfrac{f(w)}{w-z}dw}.$$
But by Cauchy's Integral Formula 
$$\dfrac{1}{2\pi i}\int_{\gamma_{2,\lambda t_0}}{\dfrac{f(w)}{w-z}dw}=0$$ for every 
$0 < \lambda \leq 1$. Similarly, 
$$\lim_{\substack{\lambda \to 0  \\ \lambda>0}}\dfrac{1}{2\pi i}
\int_{\gamma_{1,\lambda s_0}}{\dfrac{f(w)}{w-z}dw}=
\dfrac{1}{2\pi i}\int_{\gamma_1}{\dfrac{f(w)}{w-z}dw}$$. But by Cauchy's Integral 
Formula
$$\dfrac{1}{2\pi i}\int_{\gamma_1,\lambda s_0}{\dfrac{f(w)}{w-z}dw}=f(z)$$
for suitable $s_0$ and every $0<\lambda \leq 1$. Consequently
$$F(z)=f(z).$$ 
A similar argument shows that $F(z)=f(z)$ for every z in the interior of the curve
$\gamma_2$. Finally, it follows from the continuity of $F,f$ that,
if $z\in E(z_0,r)\cap L$, then $F(z)=f(z)$. 

ii) If L is the boundary of a convex body, then its projection to the x-axis is 
a compact interval $[a,b], a<b$. A vertical line $x=x_0$, $a<x_0<b$ intersects L at
two points and in this way two continuous functions $g,h$ on $(a,b)$ are defined,
where $g(x)\leq h(x)$ for $x\in (a,b)$. Both $g,h$ can be extended continuously 
at $a,b$ and their graphs have finite length, since $g$ is convex and $h$ concave. 
The vertical lines $x=a$ and $x=b$ intersect L at a point or at a vertical segment. 
For the points $(x_0,g(x_0))$ and $(x_0,h(x_0))$, where $a<x_0<b$, case i) holds
and thus $f$ will be holomorphic at those points.
If L contains a vertical segments, then the proof of case i) can be used to prove 
the desired for the interior points of the segment. It remains a finite number of
points, at least 2 and up to 4. By Riemann's theorem on removable isolated 
singularities $f$ will also be continuous at those points.
Therefore, $f$ will be holomorphic at $\Omega$.

iii) It follows from cases i) and ii).

iv) For an analytic Jordan curve there exists
a conformal function which maps the curve to a circle and an open neighbourhood of
the curve to an open neighbourhood of the circle. The same holds true for a compact 
simple analytic arc and an arc. The result follows now from case iii).

v) Obvious.
\end{proof}

The previous conditions are sufficient conditions. M. Papadimitrakis informed us
that a sufficient and necessary condition is that the continuous analytic capacity
$a(L)$ [5] is equal to $0$.
 
\section{The Fourier method on the circle and non extendability}

The basic result of this section is that generically at  
$C^k(T)$ every function is not extended holomorphically both
at the interior and the exterior of the unit circle $T$.   
Another result is that generically in $C^k(T)$ 
every function is not holomorphically extendable at any open disk
intersecting the unit circle $T$. 
We note that this is the first time it is proved that the phenomenon
of holomorphic extension from one side is topologically rare.
For one-sided holomorphic extensions we refer to
[4] and at the articles contained at its bibliography.\\

Let $\Omega\subset\mathbb{C},\Omega\neq \mathbb{C}$ be a domain and 
$z_0\in \partial\Omega$ for which there is a $\delta>0$ such that every open set 
$D(z_0,r)\cap \Omega,r\leq \delta$ is connected (*).
The space of holomorphic functions at $\Omega$ is 
denoted by $ H(\Omega)$. The topology of $H(\Omega)$
is defined by the uniform convergence on compact subsets. The space 
$H(\Omega)$ is now a Frechet space. From now on, the symbols $\Omega,z_0$ are as 
above, unless something else is stated.

\begin{defi}
Let $\Omega \subset \mathbb{C},\Omega\neq \mathbb{C}$ be a domain, 
$f\in H(\Omega)$ and $z_0\in\partial\Omega$.
The function $f$ belongs to the class $U_1(\Omega,z_0)$ if there is no open 
disk $D(z_0,r),r>0$ and there is no holomorphic function $F:D(z_0,r) 
\rightarrow \mathbb{C}$ such that $F|_{\Omega\cap D(z_0,r)}=
f|_{\Omega\cap D(z_0,r)}$.
\end{defi}

If $\Omega \subset \mathbb{C},\Omega\neq \mathbb{C}$ is a domain, 
$f\in H(\Omega)$, with $R_\zeta\in(0,+\infty)$ we will denote 
the radius of convergence of the Taylor development of $f$ of center 
$\zeta$, which always satisfies $R_\zeta\geq dist(\zeta,\partial\Omega)$. 

Our purpose is to characterize the class $U_1(\Omega,z_0)$.  

\begin{prop}
Let $\Omega\subset\mathbb{C},\Omega\neq \mathbb{C}$ be a domain, 
$z_0\in \partial\Omega$ such that (*) holds and $f \in H(\Omega)$. 
Then the following are equivalent:
\begin{enumerate}[(i)]
\item $f\in U_1(\Omega,z_0)$
\item for every $w$ in $\Omega\cap D(z_0,\delta)$ the radius of 
convergence $R_w\leq |w-z_0|$.
\item for every $z$ in a dense subset $A$ of $\Omega\cap D(z_0, \delta)$  
the radius of convergence $R_z\leq |z-z_0|$. 
\item for some $z_n\in \Omega\cap D(z_0,\delta),n=1,2,3,...$ converging to 
$z_0$ the radius of convergence $R_{z_n}\leq |z_n-z_0|$.
\end{enumerate}
\end{prop}
\begin{proof}

$(i)\Rightarrow (ii)$. We assume that $R_w>|w-z_0|$ for some
$w\in \Omega\cap D(z_0,\delta)$ and we notice that $z_0\in D(w,R_w)$.
Since $\Omega\cap D(z_0,\delta)$ is open there exists a $s>0$ such 
that $D(z_0,s)\subset  D(w,R_w) $.
Let $G$ be the holomorphic function in $D(w,R_w)$ defined by 
the Taylor series of $f$ at $w$. Then $G$ coincides
with $f$ at the open disk $D(w,dist(w,\partial\Omega))$ and by analytic continuation 
coincides with $f$ in $D(z_0,s)\cap \Omega$, because $D(z_0,s)\cap \Omega$ is also 
connected. Besides, $G$ is 
holomorphic in $D(z_0,s)$ and therefore $f$ 
does not belong to the class $U_1(D,z_0)$, which is absurd.

$(ii)\Rightarrow(i)$.If $(i)$ does not hold, 
then there is an open disk
$D(z_0,s),0<s\leq \delta$ and a holomorphic function
$F:D(z_0,s) \rightarrow \mathbb{C}$
such that $F_{/\Omega\cap D(z_0,s)}=f_{/\Omega\cap 
D(z_0,s)}$. Let $t$ be a point in $\Omega$ such that
$t\neq z_0,|t-z_0|<\dfrac{s}{2}$,
then, since $F^{(l)}(t)=f^{(l)}(t)$ for every $l=0,1,2,...$ the 
Taylor Series of center $t$ of $F$ coincides with 
the Taylor Series of center $t$ of $f$.
Hence $R_t\geq dist(t,\partial D(z_0,s))
\geq s-\dfrac{s}{2}=\dfrac{s}{2}$, but on the other hand
 $R_t=dist(t,\partial \Omega)<\dfrac{s}{2}$, which is absurd.

$(ii)\Rightarrow(iii)$.Obvious.

$(iii)\Rightarrow(i)$. Same as the direction $(ii)\Rightarrow(i)$, because 
for every $s>0$ there is a $z\in A$ such that $|z-z_0|<\dfrac{s}{2}$.

$(ii)\Rightarrow(iv)$.Obvious.

$(iv)\Rightarrow(i)$. Same as the direction
$(ii)\Rightarrow(i)$, because, if $z_n\in \Omega\cap D(z_0,\delta)$ 
converges to $z_0$, then for every $s>0$ there is an $m \in 
\left\{1,2,3,...\right\}$ such that $|z_m-z_0|<\dfrac{s}{2}$.
\end{proof}

\begin{defi}
Let $\Omega\subset\mathbb{C},\Omega\neq \mathbb{C}$ be a domain, 
$z_0\in \Omega$ and $z\in \Omega$. A function
$f \in H(\Omega)$ belongs to the class $F(\Omega,z,z_0)$ if 
the radius 
of convergence of the Taylor
expansion of f of center $z$ is less than or equal to $|z-z_0|$.
\end{defi}

We will now try to find a nice description for the class 
$F(\Omega,z,z_0)$, but previously we remind from [8] that if 
$\Omega\subset\mathbb{C},\Omega
\neq \mathbb{C}$ is a domain and
$z\in \Omega$ then the set of functions which are 
holomorphic exactly in $\Omega$ 
is a dense and $G_\delta$ subset of $H(\Omega)$.

\begin{prop}
Let $\Omega\subset\mathbb{C},\Omega\neq \mathbb{C}$ be a domain, 
$z_0\in \partial\Omega$ such that (*) holds and $z\in \Omega\cup D(z_0,\delta)$. 
Then the class $F(\Omega,z,z_0)$
is a dense and $G_\delta$ in $H(\Omega)$.
\end{prop}
\begin{proof}
Since $z_0\in\partial\Omega$,
if a function $f\in H(\Omega)$ is holomophic  exactly in 
$\Omega$ then the radius 
of convergence of the Taylor
expansion of f with center $z$ is exactly
equal to $dist(z,\partial \Omega)$ and thus
it is also less than or equal to 
$|z-z_0|$. Therefore, the above set is subset of $F(\Omega,z,z_0)$
and thus $F(\Omega,z,z_0)$ is dense in $H(\Omega)$.

Let $d_z =|z-z_0|>0$. Since the radius of convergence
$$R_z(f) =\dfrac{1}{\limsup\limits_n \sqrt[n]{\dfrac{|f^{(n)}(z)|}
{n!}}},$$ we can deduce that
$$F(\Omega,z,z_0)=\bigcap\limits_{s=1}^\infty \bigcap\limits_
{n=1}^\infty \bigcup\limits_{k=n}^\infty
\left\{f\in H(\Omega): \dfrac{1}{d_z}-\dfrac{1}{s}<
\sqrt[k]{\dfrac{|f^{(k)}(z)|}
{k!}}\right\}.$$
From the continuity of the map $H(\Omega)\ni f\rightarrow
f^{(k)}(z)\in \mathbb{C},k=1,2,...$ it follows that 
$F(\Omega,z,z_0)$ is $G_\delta$ and the proof is complete.
\end{proof}

We now have all the prerequisites to prove the next theorem:

\begin{theorem}
Let $\Omega \subset \mathbb{C},\Omega\neq \mathbb{C}$ be a domain 
and $z_0\in\partial\Omega$ such that (*) holds.
Then the set of functions which are non extendable at $z_0$, 
$U_1(\Omega,z_0)$,
is a dense and $G_\delta$ subset of $H(\Omega)$.
\end{theorem}
\begin{proof}
Let $A=(\mathbb{Q}+i\mathbb{Q})\cap \Omega$, which is a countable
dense subset of $\Omega$. From Proposition 3.2
the set $\bigcap\limits_{\zeta\in A}F(\Omega,\zeta,z_0)=U_1(\Omega,z_0)$
and from Baire's theorem it is a dense and $G_\delta$ subset of $H(\Omega)$.
\end{proof}

We will now use Theorem 3.5 to prove some nice results.

The open unit disk and the unit circle will be denoted by
$$D=\left\{z\in \mathbb{C}:|z|<1\right\}$$ and
$$T=\left\{e^{it}: t\in \bbR\right\}$$ respectively.

\begin{remark}
We can easily see that for $\Omega=D,z_1\in D$ and 
$\Omega=\mathbb{C}\setminus \overline D,z_2\in \mathbb{C}\setminus 
\overline D$ the open sets $D(z_1,r_1)\cap D$ and 
$D(z_1,r_1)\cap \mathbb{C}\setminus 
\overline D$ are connected. Thus Proposition 3.2 holds true
for the domains $D,\mathbb{C}\setminus \overline D$ and
for every $z_0 \in T.$
\end{remark}

A function $u:T \rightarrow \mathbb{C}$ belongs to the space $C(T)=C^0(T)$ if it
is continuous. The topology of $C(T)$
is defined by the seminorms
$$\sup\limits_{\theta \in \mathbb{R}}\left|\dfrac{d^lf(e^{i\theta})}
{d{\theta}^l}\right|$$
The same topology can be defined by the seminorms
$$\sup\limits_{\theta \in \mathbb{R}}\left|\dfrac{d^lf(e^{i\theta})}
{d{(e^{i\theta})^l}}\right|$$

A function $u:T \rightarrow \mathbb{C}$ belongs to the
space $C^k(T), k=1,2,...$ if it is k times continuously 
differentiable with respect to $\theta\in\bbR$. The topology of $C^k(T)$
is defined by the seminorms
$$\sup\limits_{\theta \in \mathbb{R}}\left|\dfrac{d^lf(e^{i\theta})}
{d{\theta}^l}\right|,l=0,1,2...,k$$
The same topology can be defined by the seminorms
$$\sup\limits_{\theta \in \mathbb{R}}\left|\dfrac{d^lf(e^{i\theta})}
{d{(e^{i\theta})^l}}\right|,l=0,1,2...,k$$
The above spaces are Banach spaces and thus Baire's Theorem is at our disposal.

A function $u:T \rightarrow \mathbb{C}$ belongs to the
space $C^\infty(T)$ if it is infinitely differentiable with
respect to $\theta\in\bbR$. The topology of $C^\infty(T)$
is defined by the seminorms
$$\sup\limits_{\theta \in \mathbb{R}}\left|\dfrac{d^lf(e^{i\theta})}
{d{\theta}^l}\right|,l=0,1,2...$$
Also, the same topology can be defined by the seminorms
$$\sup\limits_{\theta \in \mathbb{R}}\left|\dfrac{d^lf(e^{i\theta})}
{d{(e^{i\theta})^l}}\right|,l=0,1,2...$$
Equivalently the same topology 
can be both defined by the seminorms
 $$\sum\limits_{n \in \mathbb{Z}}  {|a_n|},\sum\limits_{n 
\in \mathbb{Z}\setminus \left\{0\right \}}
|n|^l|a_n|, l=1,2,3,...,$$ and the seminorms
$$\sup\limits_{n \in \mathbb{Z}}|a_n|,\sup\limits_{n \in \mathbb{Z}
\setminus \left\{0\right \}}|n|^l|a_n|,l=1,2,3,...,$$ where 
$$a_n=\hat{u}(n)=\dfrac{1}{2\pi}\int_{0}^{2\pi}{f(e^{i\theta})e^{-in
\theta}d\theta}$$ is the Fourier coefficient of $f \in C^\infty(T)$. 
Moreover, one can easily see that a continuous
function $u:T \rightarrow \mathbb{C}$ 
is of class $C^\infty(T)$ if and only if its 
Fourier coefficients $a_n$,$n \in \mathbb{Z}$ satisfy
$P(|n|)a_n \rightarrow 0$ for every polynomial $P$
when $n$ approaches $\pm \infty$. 
An easy result of the above is that
trigonometric polynomials are dense in $C^\infty(T)$,
since $$\sum\limits_{n=-N}^N{a_ne^{in\theta}}\xrightarrow[\
N\to \infty]{} \sum\limits_{n=-\infty}^\infty{a_ne^{in\theta}}$$ for every 
$f(e^{i\theta})=\sum\limits_{n=-\infty}^\infty{a_ne^{in\theta}}\in C^\infty(T)$.

Another interesting space is the space of holomorphic functions at D, for which 
every derivative extends continuously at T, which will be denoted as $A^\infty(D)$.
The topology of $A^\infty(D)$ is defined by the seminorms
$$\sup\limits_{z\in D}\left| {\dfrac{d^lf(z)} {dz^l}}\right|,l=0,1,2,...,$$
Equivalently, it can be defined by the seminorms
$$\sup\limits_{\theta \in \bbR}\left|{\dfrac{d^lf(e^{i\theta})}{d\theta^l}}
\right|,l=0,1,2,...$$ If $$f(z)=\sum\limits_{n=0}^\infty {a_nz^n},|z|<1,$$ then
the topology of $A^\infty(D)$ is also
induced by the seminorms $$\sum\limits_{n=0}^\infty  {|a_n|},\sum
\limits_{n=0}^\infty |n|^l|a_n|,l=1,2,3,...$$ 
In addition, the same topology is defined by the
seminorms $$\sup\limits_{n=0,1,...}|a_n|,\sup\limits_{n =1,2,...}
|n|^l|a_n|,l=1,2,3,...$$ 
Also, as above, if $f(z)=\sum\limits_{n=0}^
\infty{a_nz^n},|z|<1$, then $f$ belongs to $A^\infty(D)$
if and only if $P(n)a_n \rightarrow 0$, as $n \rightarrow \infty$
for every polynomial $P$. It follows that a continuous function
$u:T\rightarrow \mathbb{C}$ is extended as $f\in A^\infty(D)$ if and only if
$\hat{u}(n)=0 $ $\forall n<0$ and $u\in C^\infty(T)$. Since
$C^\infty(T)$ and $A^\infty(D)$ are Frechet
spaces, Baire's theorem is at our disposal. Moreover,
it is easy to see that polynomials are dense in $A^\infty(D)$, 
although the same question is open in the case
of an arbitrary simply connected $\Omega\subset \mathbb{C}$, such that
$(\mathbb{C}\cup\left\{\infty \right\})
\setminus \overline {\Omega}$ is connected and ${\overline{\Omega}}^o=\Omega$.

At this moment, we define the classes
$A_0^\infty(D)=\left\{f\in A^\infty(D): f(0)=0\right \}$ and
$$A_0^{\infty}((\mathbb{C}\setminus \overline{D})\cup\left\{\infty\right\})=
\left\{f\in A^\infty(\mathbb{C}\setminus \overline{D}): 
\lim\limits_{z\rightarrow\infty} {f(z)}=0\right \},$$
which will be needed afterwards.

Let $u:T \rightarrow \mathbb{C}$, then the functions
$$f(z)=\sum\limits_{n=0}^\infty{a(n)z^n}$$ and 
$$g(z)=\sum\limits_{n=-\infty}^{-1}\overline{a(n)}z^{-n}$$ 
belong to the classes $A^\infty(D)$ and $\overline{A_0^\infty(D)}$ 
respectively and $u=f|_{T}+\overline{g|_{T}}$. Since 
${A^{\infty}(D)}\cap \overline{{A_{0}^{\infty}(D)}}=\left\{0\right\}$ 
we have that $$ C^{\infty}(T)= {A^{\infty}(D)}|_{T}
\oplus \overline{{A_{0}^\infty(D)}}|_{T},$$ because also
$\pi_1 : C^\infty(T) \rightarrow A^\infty(D)$,
$\pi_2 : C^\infty(T) \rightarrow A_{0}^\infty(D)$, $\pi_1(u)=f
,\pi_2(u)=g$, where $f,g$ as above, are continuous.
Lets see for example the continuity of $\pi_1$.
If $u\in C^\infty(T)$ and the sequence $u_m\in C^\infty(T)$,
$m=1,2,...$ converges to $u$, then 
$$\lim\limits_{m \rightarrow \infty}{\sum\limits_{n\in \bbZ}
 {|a_m(n)-a(n)|}}=0,$$ $$\lim\limits_{m \rightarrow \infty}{\sum
\limits_{n \in \mathbb{Z} \setminus \left\{0\right \}}
|n|^l|a_m(n)-a(n)|}=0,$$ $l=1,2,3,...$. Consequently,
$$\sum\limits_{n=0}^\infty {|a_m(n)-a(n)|}\leq\sum\limits_{n\in \bbZ}
 {|a_m(n)-a(n)|}\xrightarrow[\
m\to \infty]{} 0,$$ 
$${\sum\limits_{n=1}^\infty |n|^l|a_m(n)-a(n)|}\leq{\sum\limits_{n 
\in \mathbb{Z}\setminus \left\{0\right \}}
|n|^l|a_m(n)-a(n)|}\xrightarrow[\
m\to \infty]{} 0,$$ $l=1,2,3,...$. Thus we deduce that 
$\pi_1(u_n)$ converges to $\pi_1(u)$ and that $\pi_1$ is 
continuous. Similarly we can see that $\pi_2$ is continuous.

If we think about $h=f+\bar{g}$, we will notice that the function
$h$ is harmonic in $D$ 
and, if every differential operator 
$$L_{k,l}=\dfrac{{\partial}^l}{\partial z^l}\dfrac{{\partial}^k}
{\partial \bar{z}^k}$$ is applied on $h$, it defines
a harmonic function in $D$ which extends
continuously in $\overline{D}$. Thus the space $H_\infty(D)$ is
the space of harmonic functions 
$h: D \rightarrow \mathbb{C}$ such that $L_{k,l}(h)$ extends
continuously in $\overline{D}$.
We see that for $k\geq 1$ and $l\geq 1$ 
$L_{k,l}(h)=0$, since $h$ is harmonic.
The topology of $H_\infty(D)$ is defined 
by the seminorms $$\sup\limits_{|z|\leq 1}{|L_{k,l}(h)(z)|},k,l\geq 
0,$$ and equivalently, either by the seminorms
$$\sum\limits_{n \in \bbZ}  
{|a(n)|},\sum\limits_{n 
\in \mathbb{Z}\setminus \left\{0\right \}}
|n|^l|a(n)|,l=1,2,3,...$$ or by the seminorms $$\sup\limits_{n 
\in \mathbb{Z}}|a(n)|,\sup\limits_{n \in \mathbb{Z}
\setminus \left\{0\right \}}|n|^l|a(n)|,l=1,2,3,...$$ 
In this way, the space $H_{\infty}(D)$ becomes
Frechet space. Also, we can see that 
$$H_{\infty}(D)|_{T}= {A^{\infty}(D)}|_{T}
\oplus \overline{{A_{0}^\infty(D)}}|_{T}=C^{\infty}(T).$$

However, the aforementioned analysis is not the only interesting analysis of
$C^\infty(T)$. Let $u:T \rightarrow \mathbb{C}$, then the functions
$$f(z)=\sum\limits_{n=0}^\infty{a(n)z^n}$$ and 
$$g(z)=\sum\limits_{n=-{\infty}}^{-1}{a(n)z^n},$$
where $f,g$ belong to the classes
$A^\infty(D)$ and $A_0^{\infty}((\mathbb{C}
\setminus \overline{D})
\cup\left\{\infty\right\})$ respectively and
$u=f|_{T}+{g|_{T}}$.

Also, we notice that 
$$f(z)=\dfrac{1}{2\pi i}\int_{T}{\dfrac{u(w)}{w-z}dw},|z|<1$$ and
$$g(z)=\dfrac{1}{2\pi i}\int_{T}{\dfrac{u(1/w)}{w-1/z} dw}-\hat{u}(0),|z|>1.$$
We point out that we already know that $f$ and $g$ can be extended
continuously at $T$, so the functions $f|_{T}$ and $g|_{T}$ 
are defined as limits of the above integral expressions.

Since $A^\infty(D)\cap A_{0}^\infty((\mathbb{C}\setminus 
\overline{D})\cup\left\{\infty\right\})=\left\{{0}\right\}$ 
and the functions 
$p_1 :C^\infty(T) \rightarrow A^\infty(D)$,
$p_2 :C^\infty(T) \rightarrow A_{0}^\infty((\mathbb{C}\setminus 
\overline{D})\cup\left\{\infty\right\})$,$p_1(u)=f$,$p_2(u)=g$,
where $f,g$ as above, are continuous,
we have that $$ C^{\infty}(T)={A^{\infty}(D)}|_{T}
\oplus {A_{0}^{\infty}((\mathbb{C}
\setminus \overline{D})\cup\left\{\infty\right\})}|_{T}.$$

We return now to $C^k(T)$.

\begin{defi}
Let $z_0\in T$,$k=0,1,2,...$ or $\infty$. A function $u \in C^k(T)$ belongs to the 
class $U_2(T,z_0,k)$ if there is no pair of a domain $\Omega(z_0,r)=D(z_0,r)
\cap\left\{z\in\mathbb{C}:|z|> 1\right\}$, $r>0$ 
and a continuous function $\lambda:T\cup \Omega(z_0,r)\rightarrow 
\mathbb{C}$ which is also holomorphic at
$\Omega(z_0,r)$ and satisfies $\lambda|_{D(z_0,r)\cap T}=u|_{D(z_0,r)\cap T}.$ 
\end{defi} 

\begin{defi}
Let $z_0\in T$,$k=0,1,2,...$ or $\infty$. A function $u \in C^k(T)$ belongs to the 
class $U_3(T,z_0,k)$ if there is no pair of a domain $P(z_0,r)=D(z_0,r)\cap D,r>0$ 
and a continuous function $h:T\cup P(z_0,r)\rightarrow 
\mathbb{C}$ which is also holomorphic at
$P(z_0,r)$ and satisfies $h|_{D(z_0,r)\cap T}=u|_{D(z_0,r)\cap T}.$ 
\end{defi} 

\begin{defi}
Let $z_0\in T$,$k=0,1,2,...$ or $\infty$. A function $u \in C^k(T)$ belongs to the 
class $U_4(T,z_0,k)$ if there are neither a pair of a domain $\Omega(z_0,r)=D(z_0,r)
\cap\left\{z\in\mathbb{C}:|z|> 1\right\},r>0$ 
and a continuous function $\lambda:T\cup \Omega(z_0,r)\rightarrow 
\mathbb{C}$ which is also holomorphic at
$\Omega(z_0,r)$ and satisfies $\lambda|_{D(z_0,r)\cap T}
=u|_{D(z_0,r)\cap T}$ nor pair of a domain $P(z_0,r)=D(z_0,r)\cap 
D,r>0$ and a continuous function $h:T\cup P(z_0,r)\rightarrow 
\mathbb{C}$ which is also holomorphic at
$P(z_0,r)$ and satisfies $h|_{D(z_0,r)\cap T}=u|_{D(z_0,r)\cap T}.$  
\end{defi}

\begin{remark}
At the definitions 3.7, 3.8, 3.9, if $u\in C^\infty(T)$ automatically follows that 
every derivatives of $\lambda$ and $h$ is 
continuous at the sets $(D(z_0,r)\cap T)\cup\Omega(z_0,r)$, 
$(D(z_0,r)\cap T)\cup P(z_0,r)$ respectively.
To see that, if for a $u\in C^\infty(T),u=f|_{T}+g|_{T}$
there is an open disk $D(z_0,r)$ and $\lambda$ as in Definition 
1.7, then $f$ extends continuously at $D(z_0,r)$ and
holomorphically at $\Omega(z_0,r)$. From a well-known result $f$ 
extends holomorphically at $D\cup D(z_0,r)$ and, since every derivative of
$g$ extends continuously at $T$, every 
derivative of $\lambda$ is continuous at $(D(z_0,r)\cap T)\cup\Omega(z_0,r)$. 
Similarly we deduce the corresponding result for $h$.
\end{remark}

The classes $\bigcap\limits_{z_0 \in T} U_2(T,z_0,k)$, 
$\bigcap\limits_{z_0 \in T}U_3(T,z_0,k)$,
$\bigcap\limits_{z_0 \in T}U_4(T,z_0,k)$ will be denoted
as $U_2(T,k)$, 
$U_3(T,k)$, $U_4(T,k)$  respectively, where $k=0,1,2,...$ or 
$\infty$. 

If $u \in C^k(T) ,k\geq 2$ or $\infty$, then the functions
$$f(z)=\sum\limits_{n=0}^\infty{a(n)z^n}$$ and
$$g(z)=\sum\limits_{n=-{\infty}}^{-1}{a(n)z^n}$$
are holomorphic at $D, \mathbb{C}\setminus \overline{D}$ respectively,
continuous at T and $u=f|_{T}+g|_{T}$.
Since $g$ is holomorphic at $\mathbb{C} \setminus \overline{D}$, 
if $u$ extends continuously at $D(z_0,r)\cap
\left\{z\in \mathbb{C}:|z|\geq1\right\}$
and holomorphically at 
$D(z_0,r)\cap$ $\left\{z\in \mathbb{C}:|z|> 1\right\}$, where $z_0 \in T,r>0$ then 
$f|_T=u-g|_T$ extends holomophically at $D(z_0,r)$. 
On the other hand, it is obvious that, if $f$ extends 
holomorphically at $D(z_0,r)$, 
then $u$ extends continuously at $D(z_0,r)\cap\left\{z\in \mathbb{C}:
|z|\geq 1\right\}$ holomorphically at 
$D(z_0,r)\cap\left\{z\in \mathbb{C}:|z|> 1\right\}$. On the following theorems,
when we use $f,g$ are as above.

\begin{theorem}
Let $z_0 \in T$. The class $U_2(T,z_0,\infty)$ is a dense and $G_\delta$ 
subset of $C^\infty(T)$.
\end{theorem}

\begin{proof}
From the above and Propositions 3.2 and 3.4 
$$U_2(T,z_0)=\bigcap\limits_{l=1}^\infty \bigcap\limits_{s=1}^\infty 
\bigcap\limits_{n=1}^\infty \bigcup\limits_{k=n}^\infty
\left\{u\in C^\infty(T): \dfrac{1}{d_{z_l}}-\dfrac{1}{s}<
\sqrt[k]{\dfrac{|f^{(k)}(z_l)|}{k!}}\right\},$$ where $z_l\in D$ is a dense sequence 
of $D$.

We will at first prove that the sets $$A(l,s,k)=\left\{u\in C^\infty(T): 
\dfrac{1}{d_{z_l}}-\dfrac{1}{s}<\sqrt[k]{\dfrac{|f^{(k)}(z_l)|}{k!}}\right\}$$ 
are open. Let $u_p,p=1,2,3,...$ be a sequence at $C^\infty(T)$ converging to a 
$u \in A(l,s,k)$ at the topology of $C^\infty(T)$. Then 
\begin{align*}
|f_p(e^{i\theta})-f(e^{i\theta})|
\leq \sum\limits_{i=0}^\infty |a_p(i)-a(i)|=\sum\limits_{i=1}^\infty\dfrac{1}{i^2}
|a^{''}_p(i)-a^{''}(i)|\leq \\ \leq(\sum\limits_{i=1}^\infty\dfrac{1}{i^2})
\sup\limits_{\theta\in \mathbb{R}}\left|\dfrac{d^2(u_p(e^{i\theta})-u(e^{i\theta}))}
{d\theta^2}\right|
\xrightarrow[\
p\to \infty]{} 0
\end{align*}
where $a^{''}_p(i),a^{''}(i)$ are the Fourier coefficients of 
$\dfrac{d^2u_p(e^{i\theta})}{d\theta^2},\dfrac{d^2u(e^{i\theta})}{d\theta^2}$ 
respectively. Thus $f_p$ converges uniformly to $f$ at T and from maximum modulus 
principle, $f_p$ converges uniformly to $f$ at $\overline {D}$. Therefore, there 
exists a $n_0\in \mathbb{N}$ such that $$\dfrac{1}{d_{z_l}}-\dfrac{1}{s}<\sqrt[k]
{\dfrac{|f_p^{(k)}(z_l)|}{k!}}$$ for every $p\geq n_0$. Thus, 
$u_p\in A(l,s,k)$ for $p\geq n_0$. It follows that the set $A(l,s,k)$ is 
open at $C^\infty(T)$ and the class $U_2(T,z_0)$ is $G_\delta$ subset of 
$C^\infty(T)$.

From [8] we can prove that $U_2(T,z_0,\infty)$ is also dense at $C^\infty(T)$. 
Indeed,
let $v\in C^\infty(T)$,$v=f|_T+g|_T$ then $f$ belongs to $A^\infty(D)$ and thus 
there is a sequence of functions $f_n \in A^\infty(D),n=1,2,3,...$ which are also 
holomorphic exactly at $D$ converging to $f$ at the topology of $A^\infty(D)$. Then 
${f_n}|_T+g|_T$ converges to $v$ at the topology of $C^\infty(T)$. 
But ${f_n}|_T+g|_T \in U_2(T,z_0)$ and now we can conclude that the class
$U_2(T,z_0)$ is dense subset of $C^\infty(T)$. The proof is complete.
\end{proof}

\begin{theorem}
Let $z_0 \in T$,$k=0,1,2,...$. The class $U_2(T,z_0,k)$ is a dense and $G_\delta$ 
subset of $C^k(T)$.
\end{theorem}

\begin{proof}
For $k\geq 2$ the same procedure as at Theorem 3.11 shows that the 
class $U_2(T,z_0,k)$ is a $G_\delta$ subset of $C^k(T)$. The space $C^\infty(T)$ is 
dense at every $C^k(T)$. To see that, let $u\in C^k(T)$. Then the function
$\dfrac{du^k}{{d\theta}^k}$ is continuous and thus there exists a sequence 
$u_n \in C^\infty(T)$ converging to $\dfrac{du^k}{{d\theta}^k}$. We can assume that
$a_0(u_n)=0$, since 
$a_0(u_n) \to a_0\left(\dfrac{du^k}{{d\theta}^k}\right)=0$. 
Then the sequence of functions
$$U_n(e^{it})= \int_0^t{u_n(e^{i\theta})d\theta}+c$$ converges to 
$$\int_0^t{\dfrac{du^k(e^{i\theta})}{{d\theta}^k}}d\theta+c=
\dfrac{du^{k-1}(e^{it})}{{d\theta}^{k-1}}$$
for a constant $c\in\mathbb{C}$. If we repeat this procedure we will prove the 
desired.

The class $U_2(T,z_0,\infty)$ is dense at
$C^\infty(T)$ and the class $U_2(T,z_0,k)$ contains the class 
$U_2(T,z_0,\infty)$. Therefore,
the class $U_2(T,z_0,k)$ is a dense subset of $C^k(T)$.

Let $C^k_0(T),k=0,1,2,..$ be the set of functions $u\in C^k(T)$ such that $a_0=0$. 
Then the above procedure shows that the class $U_2(T,z_0,2)\cap C^2_0(T)$ is a 
$G_\delta$ subset of $C^2_0(T)$. To see that it is also a dense subset of 
$C^2_0(T)$, let $v\in C_0^2(T)$,$v=f|_T+g|_T$. 
Then $f$ belongs to $A^\infty(D)$ and thus 
there is a sequence of functions $f_n \in A^\infty(D),n=1,2,3,...$ which are also 
holomorphic exactly at $D$ converging to $f$ at the topology of $A^\infty(D)$. The
sequence 
${f_n}|_T+g|_T$ converges to $v$ at the topology of $C^\infty(T)$. 
But the sequence ${f_n}|_T+g|_T -a_0(f_n)$ also converges to $v$ and belongs to 
$U_2(T,z_0)\cap C_0^2(T)$ and now we can conclude that the class
$U_2(T,z_0)\cap C_0^2(T)$ is dense subset of $C_0^\infty(T)$. The function 
$\varphi:C^2_0(T)\rightarrow C^1_0(T)$,
$\varphi(u)=\dfrac{du}{d\theta}$ is continuous, linear, 1-1 and onto
$C^1_0(T)$. From the open mapping theorem its inverse is also continuous.
It is easy to see that 
\begin{equation}
\varphi(U_2(T,z_0,2)\cap C^2_0(T))=U_2(T,z_0,1)\cap C^1_0(T),
\end{equation}
which combined with
the nice properties of $\varphi$ and the fact that the class 
$U_2(T,z_0)\cap C_0^2(T)$
is a dense $G_\delta$ subset of $C_0^2(T)$ indicates that the class 
$U_2(T,z_0,1)\cap C^1_0(T)$ is a dense $G_\delta$ subset of $ C^1_0(T)$.

The function $F:C^1(T)\rightarrow C^1_0(T), F(u)=u-a_0$ is continuous, linear and 
onto $C^1_0(T)$ and thus open map. Also, 
$$U_2(T,z_0,1)=F^{-1}(U_2(T,z_0,1)\cap C^1_0(T)).$$ Since $F$ is continuous, the
class $U_2(T,z_0,1)$ is $G_\delta$ subset of $C^1(T)$ and since $F$ is an open map, 
it is also dense at $C^1(T)$. By applying the last method once again, it can be 
proved that the class $U_2(T,z_0,0)$ is a dense and $G_\delta$ subset of $C(T)$.
\end{proof}

\begin{remark}
In order to prove (1) we use the fact that if $f:D \rightarrow \mathbb{C}$ is 
holomorphic and I is an open arc in $T$ and f extends continuously on $D\cup I$ and
if $f|_I \in C^1(I)$, then the derivative $\dfrac{df}{d\theta}$ extends continuously
on $D\cup I$ as well. This fact is true and its proof is complex analytic. Thus,
Fourier method works for $C^\infty$ and $C^p(T)$, $p\geq 2$ without a complex 
analytic help, while for $p=0,1$ we need something from complex analysis.
\end{remark}

\begin{theorem}
Let $z_0 \in T$, $k=0,1,2,...$ or $\infty$. The class $U_3(T,z_0,k)$ is a dense and
$G_\delta$ subset of $C^k(T)$. 
\end{theorem}

\begin{proof}
Let $\varphi :C^\infty(T)\rightarrow C^\infty(T)$,$\varphi(f)=F$
where $F: T\rightarrow \mathbb{C}$,$F(z)=f(1/z)$ for $z\in T$.
Then obviously $\varphi$ is 1-1,onto $C^\infty(T)$,continuous and 
$\varphi={\varphi}^{-1}$. Also, $u\in U_2(T,\overline{z_0})$ if and only if $
\varphi(u)\in U_3(T,z_0)$, and thus $\varphi(U_2(T))=U_3(T)$. Since the class 
$U_2(T,z_0)$ is a dense and $G_\delta$ subset of $C^\infty(T)$, it follows from the 
above 
that $U_3(T,z_0)$ is a dense and $G_\delta$ subset of $C^\infty(T)$. At the same way
it can be proved that the class $U_3(T,z_0)$ is a dense and $G_\delta$ subset of 
$C^k(T)$.
\end{proof}

Since $U_4(T,z_0,k)=U_3(T,z_0,k)\cap U_2(T,z_0,k)$ and the classes $U_2(T,z_0,k)$ 
and $U_3(T,z_0,k)$ are dense and $G_\delta$ subsets of $C^k(T)$, it follows from 
Baire's theorem the next theorem:

\begin{theorem}
Let $z_0 \in T$, $k=0,1,2,...$ or $\infty$. The class $U_4(T,z_0,k)$ is a dense and
$G_\delta$ subset of $C^k(T)$.
\end{theorem}

We will now define another class, which we will prove that shares the same 
properties with the above classes.

\begin{defi}
Let $z_0 \in T$, $k=0,1,2,...$ or $\infty$. A function $u\in C^k(T)$ belongs to 
the class $U_5(T,z_0,k)$ if 
there is no pair of disk $D(z_0,r),r>0$ and holomorphic function
$f:D(z_0,r)\rightarrow \mathbb{C}$ such that 
$f|_{D(z_0,r)\cap T}=u|_{D(z_0,r)\cap T}$.
\end{defi}

We will use the notation $U_5(T,k)$ for the class 
$\bigcap\limits_{z_0\in T} U_5(T,z_0,k)$.

We can easily see now that the class $U_5(T,z_0,k)$ 
contains $U_2(T,z_0,k)$, which is a dense and $G_\delta$ set.
Although this is a satisfying result, we will prove something even stronger.

\begin{theorem}
Let $z_0 \in T$, $k=0,1,2,...$ or $\infty$. Then class $U_5(T,z_0,k)$ 
is a dense and $G_\delta$ subset of $C^k(T)$.
\end{theorem}

\begin{proof}
We will first begin with $C^\infty(T)$. A function $u\in C^\infty(T)$,
$u=f_{|T}+g_{|T}$ belongs to the class $U_5(T,z_0,\infty)$ if and only if 
$f$ belongs to $U_1(D,z_0)$ and $g$ to $U_1(\mathbb{C}\setminus \overline{D},z_0)$.
From the proof of Proposition 3.2 we can see that there exists an open disk 
$D(z_0,r)$ such that, for every $z\in D(z_0,r)\cap D$ $f$ does not belong to the 
class $F(D,z,z_0)$ and for every 
$w\in D(z_0,r)\cap (\mathbb{C}\setminus \overline D)$ $g$ does not belong to the 
class $F(\mathbb{C}\setminus \overline{D},w,z_0)$.
Let $z_l\in D,l=1,2,3,...$,$w_l\in \mathbb{C}\setminus
\overline D,l=1,2,3,... $ be sequences converging to $z_0$.
Since both $z_l,w_l$ converge to $z_0$, there is an $m=1,2,3,...$ such that $f$ does 
not belong to $F(D,z_m,z_0)$ and $g$ does not belong to $F(\mathbb{C}\setminus 
\overline{D},w_m,z_0)$. 
Thus, $${U_5(T,z_0,\infty)}^\mathsf{c}=\bigcup\limits_{l=1}^\infty \bigcup
\limits_{s=1}^\infty \bigcup\limits_{n=1}^\infty 
\bigcap\limits_{k=n}^\infty E(l,s,k), $$
where $$E(l,s,k)=\left\{u\in C^\infty(T): \dfrac{1}
{d_{z_l}}-\dfrac{1}{s}\geq \sqrt[k]{\dfrac{|f^{(k)}(z_l)|}{k!}}, \dfrac{1}{d_{w_l}}-
\dfrac{1}{s}\geq\sqrt[k]{\dfrac{|g^{(k)}(w_l)|}{k!}}\right\}.$$
The method used at Theorem 3.11 proves that the class
${U_5(T,z_0,\infty)}^\mathsf{c}$ is $F_\sigma$ subset of $C^\infty(T)$ with empty 
interior. Equivalently the class $U_5(T,z_0)$ is a dense and $G_\delta$ subset of 
$C^\infty(T)$. Continuing as at Theorem 3.12, we can prove that the class
$U_5(T,z_0,k)$ has the desirable property. The proof is complete. 
\end{proof}

\begin{theorem}
Let $k=0,1,2,...$ or $\infty$. The classes $U_2(T,k),$ $U_3(T,k),$ $U_4(T,k)$ and
$U_5(T,k)$ are dense and $G_\delta$ subsets of $C^k(T)$.
\end{theorem}

\begin{proof}
Let $A$ be a countable dense subset of $T$. Then it holda that
$U_2(T,k)=\bigcap\limits_{z_0\in A} U_2(T,z_0,k)$, 
$U_3(T,k)=\bigcap\limits_{z_0\in A} U_3(T,z_0,k)$, 
$U_4(T,k)=\bigcap\limits_{z_0\in A} U_4(T,z_0,k)$ and 
$U_5(T,k)=\bigcap\limits_{z_0\in A} U_5(T,z_0,k)$ 
and Baire's Theorem completes the proof.
\end{proof}

\section{Non extendability from any side of an analytic curve}

\subsubsection*{The closed arc}

Let $I=\left\{e^{it}:a\leq t \leq b 
\right\}$,$0\leq a<b\leq 2\pi$ be a compact arc of T.
A function $u:I\rightarrow \mathbb{C}$ belongs to the space $C^k(I),k=0,1,2,...$
or $\infty$, 
if it is continuous and k times continuously differentiable. The topology
of $C^k(I)$ is defined by the seminorms
$$\sup\limits_{\theta \in [a,b]}|\dfrac{d^lf(e^{i\theta})}
{d{\theta}^l}|,l=0,1,2...,k$$

We can prove from Borel's theorem 
[6] that every $f\in C^\infty (I)$ extends at $C^\infty (T)$.
Indeed, if $$p_n=\dfrac{1}{n!}\dfrac{d^{n}f(e^{ia})}
{d{\theta}^{n}},q_n=\dfrac{1}{n!}\dfrac{d^{n}f(e^{ib})}
{d{\theta}^{n}},n=0,1,2,...$$ we can find $g:[b,a+2\pi]\rightarrow \mathbb{C}$ which
is infinitely differentiable at $[b,a+2\pi]$ 
such that $g^{(n)}(b)= q_n n!,g^{(n)}(a+2\pi)=p_n n!,n=0,1,2,...$.
Then the function 
$$ F(e^{it}) = \left\{
\begin{array}{ c l }
f(e^{it}),   &    t\in [a,b]  \\
g(t),        &    t\in [b,a+2\pi]  \\
\end{array}
\right. $$ 
is well defined and belongs to $C^\infty(T)$.

\begin{defi}
Let $z_0 \in I$, $k=0,1,2,...$ or $\infty$. A function $u \in C^k(I)$ belongs to the 
class $U_2(I,z_0,k)$ if there is no pair of a domain $\Omega(z_0,r)=D(z_0,r)
\cap\left\{z\in\mathbb{C}:|z|> 1\right\},z_0\in I,r>0$ 
and a continuous function $\lambda:I\cup \Omega(z_0,r)\rightarrow 
\mathbb{C}$ which is also holomorphic at
$\Omega(z_0,r)$ and satisfies $\lambda|_{D(z_0,r)\cap I}=u|_{D(z_0,r)\cap I}.$ 
\end{defi} 

\begin{defi}
Let $z_0 \in I$, $k=0,1,2,...$ or $\infty$. A function $u \in C^k(I)$ belongs to the 
class $U_3(I,z_0,k)$ if there is no pair of a domain $P(z_0,r)=D(z_0,r)\cap D,r>0$ 
and a continuous function $h:I\cup P(z_0,r)\rightarrow 
\mathbb{C}$ which is also holomorphic at
$P(z_0,r)$ and satisfies $h|_{D(z_0,r)\cap I}=u|_{D(z_0,r)\cap I}.$ 
\end{defi} 

\begin{defi}
Let $z_0 \in I$, $k=0,1,2,...$ or $\infty$. A function $u \in C^k(I)$ belongs to the 
class $U_4(I,z_0,k)$ if there are neither a pair of a domain $\Omega(z_0,r)=D(z_0,r)
\cap\left\{z\in\mathbb{C}:|z|> 1\right\},r>0$ and a continuous 
function $\lambda:I\cup \Omega(z_0,r)\rightarrow \mathbb{C}$ which is also 
holomorphic at $\Omega(z_0,r)$ and satisfies 
$\lambda|_{D(z_0,r)\cap T}=u|_{D(z_0,r)\cap I}$ 
nor a pair of a domain $P(z_0,r)=D(z_0,r)\cap D,r>0$ 
and a continuous function $h:I\cup P(z_0,r)\rightarrow 
\mathbb{C}$ which is also holomorphic at
$P(z_0,r)$ and satisfies $h|_{D(z_0,r)\cap I}=u|_{D(z_0,r)\cap I}.$  
\end{defi}

\begin{defi}
Let $z_0 \in I$, $k=0,1,2,...$ or $\infty$. A function $u\in C^k(I)$ belongs to 
the class $U_5(I,z_0,k)$ if 
there is no pair of disk $D(z_0,r),r>0$ and holomorphic function
$f:D(z_0,r)\rightarrow \mathbb{C}$ such that 
$f|_{D(z_0,r)\cap I}=u|_{D(z_0,r)\cap I}$
\end{defi}

\begin{remark}
At the definitions 4.1, 4.2, 4.3, if $u\in C^\infty(I)$, automatically follows 
that every derivative of $\lambda$ and $h$  
extends continuously at the sets $(D(z_0,r)\cap I)\cup\Omega(z_0,r)$ and
$D(z_0,r)\cap I)\cup P(z_0,r)$ respectively. This follows easily using Remark 3.10.
\end{remark}

We will use the notation $U_2(I,k),U_3(I,k),U_4(I,k)$ and $U_5(I,k)$
for the classes $\bigcap\limits_{z_0\in I} U_2(I,z_0,k)$, 
 $\bigcap\limits_{z_0\in I} U_3(I,z_0,k)$,$\bigcap\limits_{z_0\in I}
U_4(I,z_0,k)$ and $\bigcap\limits_{z_0\in I} U_5(I,z_0,k)$.

\begin{theorem}
Let $z_0 \in I$, $k=0,1,2,...$ or $\infty$. The classes 
$U_2(I,z_0,k)$, $U_3(I,z_0,k)$, $U_4(I,z_0,k)$ and $U_5(I,z_0,k)$ are dense and
$G_\delta$ subsets of $C^k(I)$.
\end{theorem}

\begin{proof}
We will prove the theorem only for the class $U_2(I,z_0,\infty)$, but the 
same method can be used for the other cases too.
Let $\varphi: C^\infty(T)\rightarrow C^\infty(I)$ defined as
$\varphi(u)=u|_I$, $u\in C^\infty(T)$, then 
$\varphi$ is continuous, linear, onto $C^\infty(I)$ and 
$C^\infty(T)$,$C^\infty(I)$ are Frechet spaces.
Thus, from Michael's theorem [11],there exists a continuous function
$\psi:C^\infty(I)\rightarrow C^\infty(T)$ such that
$\varphi(\psi(u))=u$ for $u\in C^\infty(I)$.

It is easy to see that 
$$U_2(I,z_0,\infty)=\psi^{-1}(U_2(T,z_0,\infty))$$ and 
$$U_2(I,z_0,\infty)=\varphi(U_2(T,z_0,\infty)).$$
Since $U_2(T,z_0,\infty)$ is dense at $C^\infty(T)$
and $\varphi$ continuous and onto $C^\infty(I)$, $U_2(I,z_0,\infty)$ is 
dense at $C^\infty(I)$. Also, 
$$U_2(T,z_0,\infty)=\bigcap \limits_{l=1}^\infty G_n,$$ where $G_n,n=1,2,3,...$ 
are open subsets of $C^\infty(T)$, and
$$U_2(I,z_0,\infty)=\bigcap\limits_{n=1}^\infty \psi^{-1}(G_n)$$.

Since $G_n$ are open subsets of $C^\infty(T)$, from the continuity of $\psi$ we 
deduce that the sets $\psi^{-1}(G_n)$ are open in $C^\infty(I)$. We
conclude that $U_2(I,z_0,\infty)$ is a dense and $G_\delta$ subset of $C^\infty(I)$.
\end{proof}

Let $A$ be a countable dense subset of $I$. Combining the fact that the classes 
$U_2(I,z_0),U_3(I,z_0)), U_4(I,z_0)$ and $U_5(I,z_0)$ are dense and $G_\delta$ 
subsets 
of $C^\infty(I)$ with the fact that
$U_2(I,k)=\bigcap\limits_{z_0\in A} U_2(I,z_0,k)$, $U_3(I,k)=\bigcap\limits_{z_0\in 
A} U_3(I,z_0,k)$, $U_4(I,k)=\bigcap\limits_{z_0\in A} U_4(I,z_0,k)$ and
 $U_5(I,k)=\bigcap\limits_{z_0\in A} U_5(I,z_0,k)$ 
and Baire's theorem we get the next theorem :
 
\begin{theorem}
Let $k=0,1,2,...$ or $\infty$. The classes $U_2(I,k),U_3(I,k),U_4(I,k)$ and 
$U_5(I,k)$ are dense and $G_\delta$ subsets of $C^k(I)$.
\end{theorem}

\subsubsection*{The closed interval}

At next we will extend our results even more.
We consider the compact interval $L=[\gamma,\delta],\gamma<\delta$.
Let $k=0,1,2,...$ or $\infty$. The set of continuous functions $f:L\rightarrow 
\mathbb{C}$ which are k times differentiable is denoted by
$C^k(L)$. The topology of the space is defined by the seminorms
$$\sup\limits_{z\in L}{|f^{(l)}(z)|},l=0,1,2,...,k$$. In this way 
$C^k(L)$, $k=0,1,2,...$ becomes Banach space and $C^\infty(L)$ Frechet space and 
Baire's theorem is at our disposal.

\begin{defi}
Let $k=0,1,2,...$ or $\infty$. A function $u \in C^k(L)$ belongs to the class
$U_2(L,k)$ if there is no pair of a domain $\Omega(z_0,r)=D(z_0,r)
\cap\left\{z\in\mathbb{C}:Im(z)>0 \right\},z_0\in L,r>0$ 
and a continuous function $\lambda:L\cup \Omega(z_0,r)\rightarrow \mathbb{C}$ 
which is also holomorphic at
$\Omega(z_0,r)$ and satisfies $\lambda|_{D(z_0,r)\cap L}=u|_{D(z_0,r)\cap L}.$ 
\end{defi} 

\begin{defi}
Let $k=0,1,2,...$ or $\infty$. A function $u \in C^k(L)$ belongs to the class
$U_3(L,k)$ if there is no pair of a domain $P(z_0,r)=D(z_0,r)
\cap\left\{z\in\mathbb{C}:Im(z)<0 \right\},z_0\in L,r>0$ 
and a continuous function $h:L\cup P(z_0,r)\rightarrow \mathbb{C}$ 
which is also holomorphic at $P(z_0,r)$ and satisfies 
$h|_{D(z_0,r)\cap L}=u|_{D(z_0,r)\cap L}.$ 
\end{defi} 

\begin{defi}
Let $k=0,1,2,...$ or $\infty$. A function $u \in C^k(L)$ belongs to the class
$U_4(L,k)$ if there are neither a pair of a domain
$\Omega(z_0,r)=D(z_0,r)\cap\left\{z\in\mathbb{C}:Im(z)>0 \right\}$, $z_0\in I,r>0$ 
and a continuous function $\lambda:L\cup \Omega(z_0,r)\rightarrow 
\mathbb{C}$ which is also holomorphic at
$\Omega(z_0,r)$ and satisfies $\lambda_{/D(z_0,r)\cap L}
=u_{/D(z_0,r)\cap L}$ nor pair of a domain 
$P(z_0,r)=D(z_0,r)\cap \left\{z\in\mathbb{C}:Im(z)<0 \right\},r>0$ 
and a continuous function $h:L\cup P(z_0,r)\rightarrow \mathbb{C}$ 
which is also holomorphic at $P(z_0,r)$ and satisfies 
$h|_{D(z_0,r)\cap L}=u|_{D(z_0,r)\cap L}.$  
\end{defi}

\begin{defi}
Let $k=0,1,2,...$ or $\infty$. A function $u\in C^k(L)$ belongs to 
the class $U_5(L,k)$ if 
there is no pair of disk $D(z_0,r),z_0\in L,r>0$ and holomorphic function
$f: D(z_0,r)\rightarrow \mathbb{C}$ such that 
$f|_{D(z_0,r)\cap L}=u|_{D(z_0,r)\cap L}$
\end{defi}

\begin{remark}
At the definitions 4.8, 4.9, 4.10, if $u\in C^\infty(L)$, automatically follows 
that every derivative of $\lambda$ and $h$  
extends continuously at the sets $(D(z_0,r)\cap L)\cup\Omega(z_0,r)$ and
$D(z_0,r)\cap L)\cup P(z_0,r)$ respectively. This follows easily using Remark 3.10.
\end{remark}

\begin{theorem}
Let $k=0,1,2,...$ or $\infty$. The classes $U_2(L,k)$,$U_3(L,k)$,
$U_4(L,k)$ and $U_5(L,k)$ are dense and $G_\delta$ subsets of $C^k(L)$.
\end{theorem}

\begin{proof}
We consider the function $$\mu: I_0 \rightarrow L,
\mu(z)=(\delta-\gamma)\dfrac{i+1}{i-1}\dfrac{z-1}{z+1}+\gamma,$$ 
where $I_0=\left\{e^{it}: 0\leq t \leq \dfrac{\pi}{2} \right\}$, which is
continuous, 1-1 and onto $L$.
Thus the function $$\Pi:C^\infty(L) \rightarrow
C^\infty(I_0),\Pi(f)=fo\mu$$ 
is continuous, 1-1, onto $C^\infty(I_0)$ and has continuous inverse. Also,
$$U_2(L,k)={\Pi}^{-1}(U_2(I_0,k)),$$ which
combined with the nice properties of $\Pi$ and the fact that
$U_2(I_0)$ is a dense and $G_\delta$ subset of $C^\infty(I_0)$
indicates that $U_2(L,k)$ is a dense and $G_\delta$ subset of $C^k(L)$,
We can see using the above function that the same holds true for 
the classes $U_3(L,k)$, $U_3(L,k)$, $U_4(L,k)$ and $U_5(L,k)$ .
\end{proof}

\begin{remark}
The fact that the class $U_5(L,\infty)$ is a dense and $G_\delta$ subset of
$C^\infty(L)$ strengthens the result [2].
\end{remark}

\subsubsection*{The line}

Let $k=0,1,2,...$ or $\infty$. At this point we will consider the continuous 
functions $f:\mathbb{R}\rightarrow \mathbb{C}$ which are k times continuously
differentiable. This space of functions will be
denoted by $C^k(\mathbb{R})$
and becomes Frechet space with the 
topology induced by the seminorms
$$\sup\limits_{z\in [-n,n]}{|f^{(l)}(z)|},l=0,1,2,...,k, n=1,2,3,..$$ 
We can see that every function $f\in C^\infty([\gamma,\delta])$ 
extends at $C^\infty(\mathbb{R})$. Indeed, the use of 
Borel's theorem help us find
functions $$f_1\in C^\infty([\gamma-1,\gamma]),
f_2\in C^\infty([\delta,\delta+1])$$ such that 
$${f_1}^{(n)}(\gamma-1)= 0,{f_1}^{(n)}(\gamma)=
f^{(n)}(\gamma),{f_2}^{(n)}(\delta)=
f^{(n)}(\delta),{f_2}^{(n)}(\delta +1)= 0,
n=0,1,2,...$$ 
We can extend now $f$ as $$ \tilde{f}(z) = \left\{
\begin{array}{ c l }
f(z),     &    z\in [\gamma,\delta]   \\
f_1(z),   &    z\in [\gamma-1,\gamma]  \\
f_2(z),   &    z\in [\delta,\delta+1]  \\
0,        &    z\in (-\infty,\gamma]\cup[\delta,+\infty)
\end{array}\right. $$
Then the extension is well defined and belongs to
$C^\infty(\mathbb{R})$. Also, in this case both the extension
and every derivative of the extension are bounded and converge to $0$ as $z$ 
converges to $\pm \infty$.

\begin{defi}
Let $k=0,1,2,...$ or $\infty$. A function $u \in C^k(\mathbb{R})$ belongs to the 
class $U_2(\mathbb{R},k)$ if there is no pair of a domain $\Omega(z_0,r)=D(z_0,r)
\cap\left\{z\in\mathbb{C}:Im(z)>0\right\},z_0\in \mathbb{R},r>0$ 
and a continuous function 
$\lambda:\mathbb{R}\cup \Omega(z_0,r)\rightarrow \mathbb{C}$ which is also 
holomorphic at $\Omega(z_0,r)$ and satisfies 
$\lambda|_{D(z_0,r)\cap \mathbb{R}}=u|_{D(z_0,r)\cap \mathbb{R}}.$ 
\end{defi} 

\begin{defi}
Let $k=0,1,2,...$ or $\infty$. A function $u \in C^k(\mathbb{R})$ belongs to the 
class $U_3(\mathbb{R},k)$ if there is no pair of a domain $P(z_0,r)=D(z_0,r)
\cap \left\{z\in\mathbb{C}:Im(z)<0\right\},z_0\in \mathbb{R},r>0$ 
and a continuous 
function $h:\mathbb{R}\cup P(z_0,r)\rightarrow \mathbb{C}$
which is also holomorphic at
$P(z_0,r)$ and satisfies $h|_{D(z_0,r)\cap \mathbb{R}}
=u|_{D(z_0,r)\cap \mathbb{R}}.$ 

\end{defi} 

\begin{defi}
Let $k=0,1,2,...$ or $\infty$. A function $u \in C^k(\mathbb{R})$ belongs to the 
class $U_4(\mathbb{R},k)$ if there are neither a pair of a domain
$\Omega(z_0,r)=D(z_0,r)
\cap\left\{z\in\mathbb{C}:Im(z)>0\right\}$, $z_0\in \mathbb{R}$,$r>0$ 
and a continuous function $\lambda:L\cup \Omega(z_0,r)\rightarrow 
\mathbb{C}$ which is also holomorphic at
$\Omega(z_0,r)$ and satisfies $\lambda_{/D(z_0,r)\cap \mathbb{R}}
=u_{/D(z_0,r)\cap \mathbb{R}}$ nor pair of a domain 
$P(z_0,r)=D(z_0,r)\cap \left\{z\in\mathbb{C}:Im(z)<0\right\},r>0$ 
and a continuous 
function $h:\mathbb{R}\cup P(z_0,r)\rightarrow \mathbb{C}$ 
which is also holomorphic at
$P(z_0,r)$ and satisfies $h|_{D(z_0,r)\cap \mathbb{R}}
=u|_{D(z_0,r)\cap \mathbb{R}}.$  
\end{defi}

\begin{defi}
Let $k=0,1,2,...$ or $\infty$. A function $u\in C^k(\mathbb{R})$ belongs to 
the class $U_5(\mathbb{R},k)$ if 
there is no pair of disk $D(z_0,r),z_0\in \mathbb{R},r>0$ and 
holomorphic function
$f:D(z_0,r)\rightarrow \mathbb{C}$ such that 
$f|_{D(z_0,r)\cap \mathbb{R}}=u|_{D(z_0,r)\cap \mathbb{R}}$
\end{defi}

\begin{remark}
At the definitions 4.15, 4.16, 4.17, if $u\in C^\infty(\mathbb{R})$, automatically 
follows that every derivative of $\lambda$ and $h$  
extends continuously at the sets $(D(z_0,r)\cap \mathbb{R})\cup\Omega(z_0,r)$ and
$D(z_0,r)\cap \mathbb{R})\cup P(z_0,r)$ respectively. This follows easily from
Remark 3.10.
\end{remark}

\begin{theorem}
Let $k=0,1,2,...$ or $\infty$. The classes $U_2(\mathbb{R},k)$, $U_3(\mathbb{R},k)$,
$U_4(\mathbb{R},k)$ and $U_5(\mathbb{R},k)$ are dense and $G_\delta$ subsets of
$C^k(\mathbb{R})$.
\end{theorem}

\begin{proof}
We will prove the theorem only for the class $U_2(\mathbb{R},\infty)$ and the other
cases are similar. We consider the continuous linear maps
$$\varphi_{N}: C^\infty(\mathbb{R})\rightarrow 
C^\infty([-N,N]),N=1,2,3,...$$ defined as
$$\varphi_{N}(f)=f_{/[-N,N]},$$ which are also
onto $C^\infty([-N,N])$. Since $C^\infty(\mathbb{R}),
C^\infty([-N,N])$ are Frechet spaces, $\varphi_{N}$ are also 
open maps.

Then we notice that 
$$U_2(\mathbb{R},\infty)=\bigcap\limits_{N=1}^\infty
{\varphi_{N}}^{-1}(U_2([-N,N]),\infty).$$ 
Since the sets $U_2([-N,N])$ are dense subsets of $C^\infty([-N,N])$ and 
$\varphi_{N}$ are open maps, the sets ${\varphi_{N}}^{-1}(U_2([-N,N]))$ 
are dense subsets of $C^\infty(\mathbb{R})$.

Also, we have proved above that the classes
$U_2([-N,N],\infty)$ are $G_\delta$ subsets of $C^\infty([-N,N])$ respectively.
From the continuity of the functions $\varphi_N$, the classes
${\varphi_{N}}^{-1}(U_2([-N,N])$ are $G_\delta$ subsets of
$C^\infty(\mathbb{R})$.
Consequently, from Baire's theorem the class $U_2(\mathbb{R})$ is a dense and
$G\delta$ subset of $C^\infty(\mathbb{R})$. The proof is complete.
\end{proof}

The essential part of the proof of Theorem 4.20 is that the functions
$\varphi_N$ are continuous linear onto maps between 
Frechet spaces. Thus, we can define the space $C_b^k(\mathbb{R})$
as the space of functions $f\in C^k(\mathbb{R})$, where
$f^{(l)},l=0,1,2,...,k$ are bounded. $C_b^\infty(\mathbb{R})$
is a Frechet space.  
As we have proved, every function in $C^k([-n,n]),n=1,2,3,...$
can be extended at $C_b^k(\mathbb{R})$. It is not hard
to see that the functions $\varphi_N$, as defined at the proof of
Theorem 4.20 are continuous, linear, onto maps
between Frechet spaces
and thus open maps. Therefore, if we retain the same notation the corresponding
subsets of $C_b^k(\mathbb{R})$ of non-extendable functions, namely
$U_2(\mathbb{R},k),U_3(\mathbb{R},k),U_4(\mathbb{R},k)$ and 
$U_5(\mathbb{R},k)$ are
dense and $G_\delta$ subsets of 
$C_b^\infty(\mathbb{R},k)$.

Also, the subspace of $C_b^k(\mathbb{R})$,
denoted as $C_0^k(\mathbb{R})$, which contains the functions 
$f\in C^k(\mathbb{R})$, where
$\lim\limits_{z\rightarrow \pm \infty} f^{(l)}(z)=0,l=0,1,2,...,k$.
$C_0^k(\mathbb{R})$
is a Frechet space too. If we retain the same notation
the above procedure shows that the corresponding
subsets of $C_0^\infty(\mathbb{R})$ of non-extendable functions, namely
$$U_2(\mathbb{R},k),\ U_3(\mathbb{R},k),\ U_4(\mathbb{R},k)\ \text{and}\ 
U_5(\mathbb{R},k)$$ 
are dense and $G_\delta$ subsets of $C_0^k(\mathbb{R})$.

\subsubsection*{The open arc and the open interval}

We return to the cases of the arc $I=\left\{e^{it}:a<t<b \right\},
0\leq a<b \leq 2\pi$ and the interval
$L=(\gamma,\delta),\gamma<\delta$, with the only difference that now 
they are open. The space of continuous functions $f:I\rightarrow \mathbb{C}$ 
which are k times differentiable is denoted by
$C^k(T)$ and the space of continuous functions $g:L\rightarrow 
\mathbb{C}$ which are k times differentiable is denoted by
$C^k(L)$. Both spaces become Frechet 
spaces with the topology of uniform convergence on compact sets
of the first k derivatives. Since there is an homeomorphism between $I$ or $L$ and  
$\mathbb{R}$ which maps a neighborhood of $I$ or $L$ into a neighborhood of
$\mathbb{R}$, 
the corresponding subsets of $C^k(I)$ of non-extendable functions, namely
$$U_2(I,k),\ U_3(I,k),\ U_4(I,k)\ \text{and}\ U_5(I,k),$$ and 
the corresponding subsets of $C^k(L)$ of non-extendable functions,
$$U_2(L,k),\ U_3(L,k),\ U_4(L,k)\ \text{and}\ U_{5}(L,k),$$ are dense and $G_\delta$ 
subsets of $C^k(I)$ and $C^k(L)$ respectively.

Similarly to above we define the spaces $$C_b^\infty(I),\
C_b^\infty(L),\ C_0^\infty(I)\ 
\text{and}\ C_0^\infty(L),$$ which are Frechet spaces. 
Once again, if we retain the same notation, the corresponding
subsets of $C_b^k(I)$ of non-extendable functions, namely
$$U_2(I,k),\ U_3(I,k),\ U_4(I,k)\ \text{and}\ U_5(I,k)$$ are
dense and $G_\delta$ subsets of
$C_b^k(I)$, the corresponding
subsets of $C_b^\infty(L)$ of non-extendable functions, namely
$$U_2(L,k),\ U_3(L,k),\ U_4(L,k)\ \text{and}\ U_5(L,k)$$ are dense and $G_\delta$ 
subsets of $C_b^k(L)$, the corresponding
subsets of $C_b^\infty(I)$ of non-extendable functions, namely
$$U_2(I,k),\ U_3(I,k),\ U_4(I,k)\ \text{and}\ U_5(I,k)$$ are
dense and $G_\delta$ subsets of
$C_b^k(I)$ and the corresponding
subsets of $C_0^\infty(L)$ of non-extendable functions, namely
$$U_2(L,k),\ U_3(L,k),\ U_4(L,k)\ \text{and}\ U_5(L,k)$$ are dense and $G_\delta$ 
subsets of $C_b^k(L)$.

\subsubsection*{Jordan analytic curves}
According to [1], a Jordan curve $ \gamma $ is called analytic if $ \gamma(t)=\Phi(e^{it}) $ for $ t \in \mathbb{R} $ where $ \Phi : D(0,r,R)\rightarrow \mathbb{C} $ is injective and holomorphic on $ D(0,r,R)=\{ z \in \mathbb{C}: r< |z|<R \} $ with $ 0<r<1<R $.
\\
Let $ \Omega $ be the interior of a curve $ \delta $ which is Jordan analytic curve. Let $ \varphi:D\rightarrow \Omega $ be a Riemann function. Then $ \varphi $ has a conformal extension in an open neighbourhood of $ \overline{D} $, $ W $,   and $ \sup\limits_{|z|\leq1} |\varphi^{(l)}(z)|< +\infty$, $ \sup\limits_{|z|\leq1} |(\varphi^{-1})^{(l)}(z)|< +\infty$ for every $ l=0,1,2 \dots $. Let $ \gamma:[0,2\pi]\rightarrow \mathbb{C} $ be a curve with $ \gamma(t)=\varphi(e^{it}) $ for $ t \in [0, 2\pi] $. We extend $ \gamma $ to a $ 2\pi $-periodic function. At the following, when we use the symbols $ \Omega, \gamma, \varphi, W $ , they will have the same meaning as mentioned above, unless something else is denoted.

\begin{defi}
Let $ u : \partial \Omega \rightarrow \mathbb{C}$ be a continuous function. Also, let $ \omega : T\rightarrow \mathbb {C}$ be the function $ \omega=uo \varphi $. Then, for $ k=0,1,2, \dots, +\infty $ we have that $ u \in C^{k}(\partial \Omega) $ if and only if $ w \in C^{k}(T)$ and the function $ h:C^{k}(\partial \Omega)\rightarrow C^{k}(T) $ with $ h(u)=uo\varphi $ is an isomorphism. We denote that the topology of $ C^{k}(\partial \Omega) $ is defined by the seminorms 
$$ \sup\limits_{\vartheta\in \mathbb{R}} \left| \frac{d^{l}(u\circ\varphi)}{d\vartheta^{l}} (e^{i\vartheta}) \right|, l=0,1,2,3 \dots, k $$
\\
\end{defi}
Due to the isomorphism $ h $ we can easily generalize some results of the previous paragraphs. 

\begin{defi}
For $ k=0,1,2, \dots, +\infty $, a function $u \in C^{k}(\gamma^{*})$ belongs to the class
$U_5(\gamma^{*},k)$ if there is no pair of a $D(z_0,r), z_0\in \gamma^{*}$
and a continuous 
function $\lambda:\gamma^{*}\cup D(z_0,r)\rightarrow 
\mathbb{C}$ which is also holomorphic at
$D(z_0,r)$ and satisfies $\lambda_{/D(z_0,r)\cap \gamma^{*}}
=u_{/D(z_0,r)\cap \gamma^{*}}.$ We will simply write $U_5(\gamma^{*})$ for $U_5(\gamma^{*},\infty)$
\end{defi}

\begin{theorem}
For $ k=0,1,2, \dots, +\infty $, the set of the functions $ f \in C^{k}(\partial \Omega) $ which do not belong  $U_5(\gamma^{*})$ is a dense and $ G_\delta $ subset of $C^{k}(\partial \Omega)$.
\end{theorem}
\begin{proof}
We will prove that for $ k=+\infty $ and the other cases can be proven in the same way. We know that $ h:C^{\infty}(\partial \Omega)\rightarrow C^{\infty}(T) $ with $ h(u)=uo\varphi $ is an isomorphism. We will prove that $ h(U_5(\gamma^{*}))=U_5(T) $ or equivalently $ h(U_5(\gamma^{*})^c)=U_5(T)^c $. If $ f\in U_5(\gamma^{*})^c $, then there exist $ \gamma(t_0) $ , an open set $ V $ with $\gamma(t_0)\in V \subseteq \varphi(W) $ and a holomorphic function $ F:V \rightarrow \mathbb {C} $ such that $ F|_{V \cap \gamma^{*}}=f|_{V \cap \gamma^{*}} $. Therefore, $ Fo\varphi:\varphi^{-1}(V) \rightarrow \mathbb {C} $ is holomorphic and $ Fo\varphi|_{\varphi^{-1}(V) \cap T}=h(f)|_{\varphi^{-1}(V) \cap T}=fo\varphi|_{\varphi^{-1}(V) \cap T} $. As a result,  $ f\in U_5(T)^c $ and $ h(U_5(\gamma^{*}))\subseteq U_5(T) $. If $ f\in U_5(T)^c $, then there exist $ e^{it_0} $ , an open set $ V $ with $e^{it_0}\in V \subseteq W $ and a holomorphic function $ F:V \rightarrow \mathbb {C} $ such that $ F|_{V \cap T}=f|_{V \cap T} $. Thus, $ Fo\varphi^{-1}:\varphi(V) \rightarrow \mathbb {C} $ is holomorphic and $ Fo\varphi^{-1}|_{\varphi(V) \cap \gamma^{*}}=fo\varphi^{-1}|_{\varphi(V) \cap \gamma^{*}} $, $ h(fo\varphi^{-1}|_\gamma^{*})=f $. Therefore,  $ f\in h(U_5(\gamma^{*})^c) $ and $ h(U_5(\gamma^{*})\supseteq U_5(T) $. From theorem 3.18 we have that $ U_5(T) $ is a dense and $ G_\delta $ subset of $C^{\infty}(T)$ and therefore  $ U_5(\gamma^{*}) $ is a dense and $ G_\delta $ subset of $C^{\infty}(\partial \Omega)$.
\end{proof}

At this point we give some definitions which are generalizations of the definitions of the previous paragraph. 

\begin{defi}
For $ k=0,1,2, \dots, +\infty $, a function $u \in C^{k}(\gamma^{*})$ belongs to the class
$U_2(\gamma^{*},k)$ if there is no pair of a domain $\Omega(z_0,r)=D(z_0,r)
\cap B$, where $ B $ is the non-bounded connected component of $ \mathbb {C}\setminus \gamma^{*} $, $z_0\in \gamma^{*},r>0$ 
and a continuous 
function $\lambda:\gamma^{*}\cup \Omega(z_0,r)\rightarrow 
\mathbb{C}$ which is also holomorphic at
$\Omega(z_0,r)$ and satisfies $\lambda|_{D(z_0,r)\cap \gamma^{*}}
=u|_{D(z_0,r)\cap \gamma^{*}}.$ 
\end{defi}

\begin{defi}
For $ k=0,1,2, \dots, +\infty $, a function $u \in C^{k}(\gamma^{*})$ belongs to the class
$U_3(\gamma^{*},k)$ if there is no pair of a domain $P(z_0,r)=D(z_0,r)\cap 
A$, where $ A $ is the bounded connected component of $ \mathbb {C}\setminus \gamma^{*} $, $z_0\in \gamma^{*},r>0$ 
and a continuous 
function $h:\gamma^{*}\cup P(z_0,r)\rightarrow 
\mathbb{C}$ which is also holomorphic at
$P(z_0,r)\cup \gamma^{*} $ and satisfies $h_{/D(z_0,r)\cap \gamma^{*}}
=u_{/D(z_0,r)\cap \gamma^{*}}.$ 
\end{defi}

\begin{defi}
For $ k=0,1,2, \dots, +\infty $, a function $u \in C^{k}(\gamma^{*})$ belongs to the class
$U_4(\gamma^{*},k)$ if there are neither a pair of a domain $
\Omega(z_0,r)=D(z_0,r)
\cap B$, where $ B $ is the non-bounded connected component of $ \mathbb {C}\setminus \gamma^{*} $,$z_0\in \gamma^{*},r>0$ 
and a continuous 
function $\lambda:\gamma^{*}\cup \Omega(z_0,r)\rightarrow 
\mathbb{C}$ which is also holomorphic at
$\Omega(z_0,r)$ and satisfies $\lambda_{D(z_0,r)\cap \gamma^{*}}
=u_{/D(z_0,r)\cap \gamma^{*}}$ nor a pair of a domain $P(z_0,r)=D(z_0,r)\cap 
A$,where $ A $ is the bounded connected component of $ \mathbb {C}\setminus \gamma^{*} $,$r>0$ 
and a continuous 
function $h:\gamma^{*}\cup P(z_0,r)\rightarrow 
\mathbb{C}$ which is also holomorphic at
$P(z_0,r)\cup \gamma^{*} $ and satisfies $h_{/D(z_0,r)\cap \gamma^{*}}
=u_{/D(z_0,r)\cap \gamma^{*}}.$  
\end{defi}

Now, we have the generalization of the theorems 3.18. 

\begin{theorem}
For $ k=0,1,2, \dots, +\infty $, the classes $ U_2(\gamma^{*},k),U_3(\gamma^{*},k),U_4(\gamma^{*},k) $ are dense and $ G_{\delta} $ subsets of $ C^{k}(\gamma^{*}) $.
\end{theorem}
\begin{proof} 
Using the isomorphism $ h:C^{k}(\partial \Omega)\rightarrow C^{k}(T) $ with $ h(u)=uo\varphi $, the theorem can be proven in the same way with the theorem 4.23.
\end{proof}

\subsubsection*{Analytic curves}
We will now extend the above results to more general curves.

\begin{defi}
We will say that a simple curve $ \gamma:[a,b]\rightarrow \mathbb{C} $ is analytic if there is an open set $ [a,b]\subseteq V\subseteq \mathbb{C} $ and a holomorphic and injective function $ F:V\rightarrow \mathbb{C} $ such that $ F|_{[a,b]}=\gamma|_{[a,b]} $. 
 
\end{defi}

By analytic continuity,it is obvious that for an analytic simple curve and for specific $ V $ as above, there is unique $ F $ as above. For the following definitions and propositions $ \gamma $ will be an analytic curve and $ V, [a,b] $ as in the definitions above. 
\\

Let $k=0,1,2,...$ or $\infty$. The set of continuous functions $f:\gamma^{*}\rightarrow 
\mathbb{C}$ which are k times differentiable is denoted by
$C^k(\gamma^{*})$. The topology of the space is defined by the seminorms
$$\sup\limits_{t\in [a,b]}{|(fo\gamma)^{(l)}(t)|},l=0,1,2,...,k$$. In this way 
$C^k(\gamma^{*})$, $k=0,1,2,...$ becomes Banach space and $C^\infty(\gamma^{*})$ Frechet space and 
Baire's theorem is at our disposal.

\begin{defi}
For a curve $ \gamma $ as above, we define as $ K(\gamma,+) $  the open set $ F(\left\{z\in\mathbb{C}:Im(z)>0 \right\} \cap V) $ and as $ K(\gamma,-) $  the open set $ F(\left\{z\in\mathbb{C}:Im(z)<0 \right\} \cap V) $
 
\end{defi}

We can imagine the above two sets as two different sides of a neighbourhood of $\gamma$.

\begin{defi}
Let $k=0,1,2,...$ or $\infty$. A function $u \in C^k(\mathbb {\gamma^{*}})$ belongs to the class
$U_2(\gamma^{*},k)$ if there is no pair of a domain $ W $ with $ W \cap \gamma^{*} \neq  \emptyset$ and a continuous function $\lambda:\gamma^{*}\cup \left(W\cap K(\gamma,+) \right) \rightarrow \mathbb{C}$ 
which is also holomorphic at
$W\cap K(\gamma,+)$ and satisfies $\lambda|_{W \cap \gamma^{*}}=u|_{W \cap \gamma^{*}}.$ 
\end{defi} 

\begin{defi}
Let $k=0,1,2,...$ or $\infty$. A function $u \in C^k(\mathbb {\gamma^{*}})$ belongs to the class
$U_3(\gamma^{*},k)$ if there is no pair of a domain $ W $ with $ W \cap\gamma^{*}\neq\emptyset $ and a continuous function $\lambda:\gamma^{*}\cup \left(W\cap K(\gamma,-) \right) \rightarrow \mathbb{C}$ 
which is also holomorphic at
$W\cap K(\gamma,-)$ and satisfies $\lambda|_{W \cap \gamma^{*}}=u|_{W \cap \gamma^{*}}.$
\end{defi} 

\begin{defi}
Let $k=0,1,2,...$ or $\infty$.  A function $u \in C^k(\mathbb {\gamma^{*}})$ belongs to the class
$U_4(\gamma^{*},k)$ if there is neither pair of a domain $ W $ with $ W \cap \gamma^{*} \neq  \emptyset$ and a continuous function $\lambda:\gamma^{*}\cup \left(W\cap K(\gamma,+) \right) \rightarrow \mathbb{C}$ 
which is also holomorphic at
$W\cap K(\gamma,+)$ and satisfies $\lambda|_{W \cap \gamma^{*}}=u|_{W \cap \gamma^{*}}.$ nor pair of a domain $ W $ with $ W \cap\gamma^{*}\neq\emptyset $ and a continuous function $\lambda:\gamma^{*}\cup \left(W\cap K(\gamma,-) \right) \rightarrow \mathbb{C}$ 
which is also holomorphic at
$W\cap K(\gamma,-)$ and satisfies $\lambda|_{W \cap \gamma^{*}}=u|_{W \cap \gamma^{*}}.$  
\end{defi}

\begin{defi}
Let $k=0,1,2,...$ or $\infty$. A function $u\in C^k(\gamma^{*})$ belongs to 
the class $U_5(\gamma^{*},k)$ if 
there is no pair of disk $D(z_0,r),z_0\in \gamma^{*},r>0$ and holomorphic function
$f: D(z_0,r)\rightarrow \mathbb{C}$ such that 
$f|_{D(z_0,r)\cap \gamma^{*}}=u|_{D(z_0,r)\cap \gamma^{*}}$
\end{defi}

\begin{remark}
We observe that the above definitions do not depend on the choice of $ V $ at definition 4.28.
\end{remark}

\begin{theorem}
Let $k=0,1,2,...$ or $\infty$. The classes $U_2(\gamma^{*},k)$,$U_3(\gamma^{*},k)$,
$U_4(\gamma^{*},k)$ and $U_5(\gamma^{*},k)$ are dense and $G_\delta$ subsets of $C^k(\gamma^{*})$.
\end{theorem}

\begin{proof}
The function $ f: C^k(\gamma^{*})\rightarrow C^k([a,b]) $ with $ f(u)=uo\gamma $ is a topological isomorphism. Using the function $ F $ of definition 4.28, as in theorem 4.23, we can prove that $ f\left(U_5(\gamma^{*},k)\right)=U_5([a,b],k) $. Thus, $U_5(\gamma^{*},k)$ is dense and $G_\delta$ subset of $C^k(\gamma^{*})$, because $U_5([a,b],k) $ is dense and $G_\delta$ subset of $C^k([a,b])$ according to the theorem 4.13. \\
In the same way we prove this result for the classes $U_2(\gamma^{*},k)$,$U_3(\gamma^{*},k)$,
$U_4(\gamma^{*},k)$.
\end{proof}

\subsubsection*{Locally analytic curve defined on open interval}
Now we extend the results to some curves which are defined on open intervals.

\begin{defi}
We will say that a simple curve $ \gamma:(a,b)\rightarrow \mathbb{C} $ is locally analytic if for every $ t_0 \in  (a,b)$ there exists an open set $ t_0\in V\subseteq \mathbb{C} $ and a holomorphic and injective function $ F:V\rightarrow \mathbb{C} $ such that $ F|_{(a,b)}=\gamma|_{(a,b)} $. 
 
\end{defi}

For the following definitions and propositions $ \gamma $ will be a locally analytic curve and $ V, (a,b) $ as in the  above definition. 
\\

Let $k=0,1,2,...$ or $\infty$. The set of continuous functions $f:\gamma^{*}\rightarrow 
\mathbb{C}$ which are k times differentiable is denoted by
$C^k(\gamma^{*})$. The topology of the space is defined by the seminorms
$$\sup\limits_{t\in [a+\frac{1}{n},b-\frac{1}{n}]}{|(fo\gamma)^{(l)}(t)|},l=0,1,2,...,k, n=1,2,3, \dots$$. In this way 
$C^k(\gamma^{*})$, $k=0,1,2,...$ becomes Banach space and $C^\infty(\gamma^{*})$ Frechet space and 
Baire's theorem is at our disposal.

\begin{defi}
Let $k=0,1,2,...$ or $\infty$. A function $u \in C^k(\mathbb {\gamma^{*}})$ belongs to the class
$U_2(\gamma^{*},k)$ if there is no $ n\geq 1 $ such that $ u|_{\delta_n^{*}} \in U_2(\delta_n^{*},k) $, where $ \delta_n=\gamma|_{[a+\frac{1}{n},b-\frac{1}{n}]} $.
\end{defi} 

\begin{defi}
Let $k=0,1,2,...$ or $\infty$. A function $u \in C^k(\mathbb {\gamma^{*}})$ belongs to the class
$U_3(\gamma^{*},k)$ if there is no $ n\geq 1 $ such that $ u|_{\delta_n^{*}} \in U_3(\delta_n^{*},k) $, where $ \delta_n=\gamma|_{[a+\frac{1}{n},b-\frac{1}{n}]} $.
\end{defi} 

\begin{defi}
Let $k=0,1,2,...$ or $\infty$. A function $u \in C^k(\mathbb {\gamma^{*}})$ belongs to the class
$U_4(\gamma^{*},k)$ if there is no $ n\geq 1 $ such that $ u|_{\delta_n^{*}} \in U_4(\delta_n^{*},k) $, where $ \delta_n=\gamma|_{[a+\frac{1}{n},b-\frac{1}{n}]} $.  
\end{defi}

\begin{defi}
Let $k=0,1,2,...$ or $\infty$. A function $u\in C^k(\gamma^{*})$ belongs to 
the class $U_5(\gamma^{*},k)$ if 
there is no pair of disk $D(z_0,r),z_0\in \gamma^{*},r>0$ and holomorphic function
$f: D(z_0,r)\rightarrow \mathbb{C}$ such that 
$f|_{D(z_0,r)\cap \gamma^{*}}=u|_{D(z_0,r)\cap \gamma^{*}}$
\end{defi}

\begin{theorem}
Let $k=0,1,2,...$ or $\infty$. The classes $U_2(\gamma^{*},k)$,$U_3(\gamma^{*},k)$,
$U_4(\gamma^{*},k)$ and $U_5(\gamma^{*},k)$ are dense and $G_\delta$ subsets of $C^k(\gamma^{*})$.
\end{theorem}

\begin{proof}
The function $ f: C^k(\gamma^{*})\rightarrow C^k(a,b)) $ with $ f(u)=uo\gamma $ is a topological isomorphism. We will prove that $ f\left(U_5(\gamma^{*},k)\right)=U_5((a,b),k) $. If $ u\in U_5(\gamma^{*},k) $ and $ f(u) \notin U_5((a,b),k) $,then there exists $ n $ such that $ f(u)|_{[a+\frac{1}{n},b-\frac{1}{n}]} \notin U_5([a+\frac{1}{n},b-\frac{1}{n}],k) $. Then,as we have seen at a previous proof, we have that $ u|_{\gamma([a+\frac{1}{n},b-\frac{1}{n}])} \notin U_5(\gamma([a+\frac{1}{n},b-\frac{1}{n}]),k) $ and therefore $ u\in U_5(\gamma^{*},k) $, which is a contradiction. So, $ f\left(U_5(\gamma^{*},k)\right)\subseteq U_5((a,b),k) $. Similarly, we prove that  $ f\left(U_5(\gamma^{*},k)\right)\supseteq U_5((a,b),k) $ and therefore, $ f\left(U_5(\gamma^{*},k)\right)= U_5((a,b),k) $ Thus, $U_5(\gamma^{*},k)$ is dense and $G_\delta$ subset of $C^k(\gamma^{*})$, because $U_5((a,b),k)$ is dense and $G_\delta$ subset of $C^k((a,b))$ according to the respective result of the paragraph about the open interval. \\
In the same way we prove this result for the classes $U_2(\gamma^{*},k)$,$U_3(\gamma^{*},k)$,
$U_4(\gamma^{*},k)$.
\end{proof}

\begin{remark}
In the previous results nothing essential changes if we cosinder $ a=-\infty $ and $ b=-\infty $.
\end{remark}

\section{Real analyticity as a rare phenomenon}
Now, we will associate the phenomenon of non-extendability with that of real analyticity. Using results of non-extendability, we will prove results for real analyticity and we will see an application of the inverse.
\begin{prop}
Let $f:T \rightarrow \mathbb{C}$ and $t_0 \in \bbR$. Then
the following are equivalent:
\begin{enumerate}[(i)]
 \item There exists a powerseries of real variable
 $$\sum\limits_{n=0}^\infty a_n(t-t_0)^n,a_n \in \mathbb{C}$$ which
 has strictly positive radius of convergence $r>0$ 
 and there exists a $0<\delta\leq r$ such that $$f(e^{it})=
 \sum\limits_{n=0}^\infty a_n(t-t_0)^n$$ for $t\in(t_0-\delta,t_0+
 \delta)$.
 \item There exists an open set 
  $V \subset \mathbb{C},e^{it_0}\in V$ and a holomorphic function  
  $F:V
  \rightarrow 
  \mathbb{C}$ such that $F_{/V\cap T}=f_{/V\cap T}$.
 \item There exists a powerseries of complex variable 
 $$\sum\limits_{n=0}^\infty b_n(z-e^{it_0})^n,b_n \in \mathbb{C}$$ w
 which has strictly positive radius of
 convergence $s>0$ and there exists $0<\epsilon\leq s$ such that
 $$f(e^{it})=
 \sum\limits_{n=0}^\infty b_n(e^{it}-e^{it_0})^n$$ for $t\in(t_0-   
 \epsilon,t_0+
 \epsilon)$.

\end{enumerate}  
\end{prop}

\begin{proof}
$(i)\Rightarrow(ii)$
We consider
$$g(z)=\sum\limits_{n=0}^\infty a_n(z-t_0)^n,$$
$z\in D(t_0,\delta)$,
which is well defined and holomorphic at $D(t_0,\delta)$. 
We can assume that $\delta<\pi$, so that the function 
$$\varphi:D(t_0,\delta)\rightarrow \mathbb{C},\varphi(z)=e^{iz}$$
is holomorphic and 1-1 and consequently there exists the inverse
function 
$$\varphi^{-1}: \varphi(D(t_0,\delta))\rightarrow \mathbb{C}$$ 
which is also holomorphic. We define as $$F=go\varphi^{-1}:
\varphi(D(t_0,\delta))\rightarrow \mathbb{C},$$ 
which is holomorphic function and $$F(z)=f(z),z\in \varphi(D(t_0,
\delta))
\cap T.$$
$(ii)\Rightarrow(iii)$
Obvious \newline
$(iii)\Rightarrow(i)$
Let $$G(z)=\sum\limits_{n=0}^\infty b_n(z-e^{it_0})^n,z\in 
D(e^{it_0},\epsilon).$$
The function $$\varphi: \mathbb{C}\rightarrow \mathbb{C},
\varphi(z)=e^{iz}$$ is holomorphic, and as a result the function
$$Go\varphi :\mathbb{C}\rightarrow \mathbb{C}$$ is holomorphic. 
Hence, there exists $\delta>0$ such that
$D(t_0,\delta)\subset \varphi^{-1}(D(e^{it_0},\epsilon))$ and
$a_n\in \mathbb{C}$,$n=0,1,2,...$ for which
$$(Go\varphi)(z)=\sum\limits_{n=0}^\infty a_n(z-t_0)^n,
z\in D(t_0,\delta)$$ and therefore 
$$f(e^{it})=G(e^{it})=\sum\limits_{n=0}^\infty a_n(t-t_0)^n,t
\in(t_0-\delta,t_0+\delta).$$
\end{proof} 

\begin{defi}
A function $ f:T\rightarrow \mathbb{C} $ which satisfies the property (i) of Proposition 5.1
is called real analytic at $e^{it_0}$.
\end{defi}

\begin{defi}
A function $ f:T\rightarrow \mathbb{C} $ which satisfies the property (ii) of Proposition 5.1
is called extendable at $e^{it_0}$.
\end{defi}

Proposition 5.1 can be also written as following:

\begin{prop}
A function $ f:T\rightarrow \mathbb{C} $ is extendable at $e^{it_0}$ if and only if it is real analytic at $e^{it_0}$.
\end{prop}

Now, we want to generalize the previous Proposition to more general curves, in order to associate the phenomenon of non-extendability with that of real analyticity to these curves. We will use this connection in order to prove results which show that real analyticity is a rare phenomenon. To prove this connection, we need to use the Proposition 5.1 to these curves. However, we will prove that this lemma is true only for Jordan curves which are analytic. 

\begin{lemma}
Let $ t_0 \in (0,1) $ and $ \gamma:[0,1]\rightarrow \mathbb{C} $ a Jordan curve for which the following holds:\\
For every $ I\subseteq \mathbb{R} $ open interval and for every function $ f: I \rightarrow \mathbb {C} $  where $ t_0 \in I $ the followings are equivalent:\\
1) There  is a power series of real variable
 $$\sum\limits_{n=0}^\infty a_n(t-t_0)^n,a_n \in \mathbb{C}$$  
 with a positive radius of convergence $r>0$ and there is $0<\delta\leq r$ so that $$f(t)=
 \sum\limits_{n=0}^\infty a_n(t-t_0)^n$$ for $t\in(t_0-\delta,t_0+
 \delta)$.\\
2)  There  is a power series of complex variable 
 $$\sum\limits_{n=0}^\infty b_n(z- \gamma(t))^n,b_n \in \mathbb{C}$$ with a positive radius of convergence $s>0$ and $0<\epsilon\leq s$ so that 
 $$f(t)=
 \sum\limits_{n=0}^\infty b_n(\gamma(t)-\gamma(t_0))^n$$ for $t\in(t_0-   
 \epsilon,t_0+
 \epsilon)$.
 \\
 Then $ \gamma $ is locally analytic at $ t_0 $, meaning that there exists an
open set $ V $ in $ \mathbb{C} $ containing an interval $(t_0-\delta, t_0+\delta)\subseteq[0, 1]$ where $ \delta >0 $ and there exists a conformal mapping
$\varphi : V \rightarrow \mathbb {C}$ such that $\phi|_{(t_0-\delta, t_0+\delta)} = \gamma|_{(t_0-\delta, t_0+\delta)}$.
\end{lemma}

\begin{proof}
We will use the implication $ 2)\Rightarrow 1) $ only to prove that $ \gamma $ is differentiable at an open interval which contains $ t_0 $. There exists $ \beta >0 $ so that $I=(t_0-\beta, t_0+\beta)\subseteq[0, 1]$. For every $ t \in I $ we choose $ f(t)=\gamma(t)= \gamma(t_0)+(\gamma(t)-\gamma(t_0)) $ and so by the $ 2)\Rightarrow 1) $ we have that  there exists $0< \delta < \beta$ so that $$\gamma(t)=
 \sum\limits_{n=0}^\infty a_n(t-t_0)^n$$   for some constants $a_n \in \mathbb{C}$ for every $t\in(t_0-\delta,t_0+
 \delta)$. Therefore $ \gamma $ is differentiable in this interval. Now, for $ g(t)=t=t_0+(t-t_0) $ by $ 1)\Rightarrow 2)$ we have that there exists  $0<\epsilon\leq \delta$ so that 
 $$t= \sum\limits_{n=0}^\infty b_n(\gamma(t)-\gamma(t_0))^n$$ for some constants $b_n \in \mathbb{C}$ for every $t\in(t_0- \epsilon,t_0+ \epsilon)$. We differentiate the above equation at $t=t_0 $ and we have that $ 1=b_{1}\gamma'(t_0) $. Therefore $b_{1}\neq 0$. Now, because $ \gamma $ is not constant, the power series  $\sum\limits_{n=0}^\infty b_n(z-\gamma(t_0))^n$ has a positive radius of convergence and so there exists $\alpha>0$ such that $ \gamma(t)\in D(\gamma(t_0),\alpha)$ for every $ t \in (t_0- \epsilon,t_0+ \epsilon) $ and the function $ f:D(\gamma(t_0),\alpha)\rightarrow \mathbb{C}$ with $$ f(z)=\sum\limits_{n=0}^\infty b_n(z-\gamma(t_0))^n $$ is a holomorphic one. Also, we have that  $ f(\gamma(t))=t $ for every $ t \in (t_0- \epsilon,t_0+ \epsilon) $. Because $ f'(\gamma(t_0))=b_1\neq 0 $ , $ f $ is locally invertible and let $ h=f^{-1} $ in an open disk $ D(t_0,\eta) $ where $ 0<\eta <\epsilon $. Then, $ \gamma(t)=h(t) $ for every $ t \in (t_0-\eta, t_0+\eta) $ and h is holomorphic and injection and the proof is complete.
\end{proof}

\begin{remark}
The above proof shows that if $\gamma $ in lemma 5.5 belongs to $ C^1([0,1]) $, then the conclusion of the lemma is true even if we suppose only that $ 1)\Rightarrow 2) $ is true. A simpler proof of the lemma which ,however, does not prove the version of the lemma described in the previous sentence,can also be made.
\end{remark}
 
 Now, if we have a Jordan curve which satisfies the assumptions of the previous lemma at each of its point, then from this lemma we conclude that the curve is locally analytic at each of its points. From [9] we conclude that the curve is analytic Jordan curve. So, Proposition 5.1 cannot be generalized for Jordan curves which are not analytic. Below we show that for Jordan curves which are analytic the generalization is true.
 
 \begin{lemma}
 Let $f:I \rightarrow \mathbb{C}$ be a function, where $ I $ is an open interval and $t_0 \in I $ and $ \gamma:[0,2\pi]\rightarrow \mathbb{C}$ an analytic Jordan curve. We extend $ \gamma $ to a $ 2\pi $-periodic function. Then the followings are equivalent: \\
 
 1) There exists a power series with real variable
 $$\sum\limits_{n=0}^\infty a_n(t-t_0)^n,a_n \in \mathbb{C}$$ with positive radius of convergence $r>0$ and there exists a $0<\delta\leq r$ such that $$f(t)=
 \sum\limits_{n=0}^\infty a_n(t-t_0)^n$$ for $t\in(t_0-\delta,t_0+
 \delta)$.\\

 2) There exists a power series with complex variable
 $$\sum\limits_{n=0}^\infty b_n(z-\gamma(t_0))^n,b_n \in \mathbb{C}$$ with positive radius of convergence $s>0$ and there exists $\epsilon>0$ such that 
 $$f(t)=
 \sum\limits_{n=0}^\infty b_n(\gamma(t)-\gamma(t_0))^n$$ for $t\in(t_0-   
 \epsilon,t_0+
 \epsilon)$.
\end{lemma}
\begin{proof} 
Because $\gamma$ is an analytic Jordan curve, there is a holomorphic and injective function $ \Phi :D(0,r,R)\rightarrow\mathbb{C} $ where $ 0<r<1<R $ and $ \gamma(t)=\Phi(e^{it}) $ for every $ t \in \mathbb{R} $. Thefore there is an open disk $ D(t_0,\varepsilon)\subseteq\mathbb{C} $, where $ \varepsilon>0 $ and a holomorphic and injective function $ \Gamma:D(t_0,\varepsilon)\rightarrow \mathbb{C} $ with $\Gamma(z)=\Phi(e^{iz})$.\\

$(i)\Rightarrow(ii)$
We consider the function 
$$g(z)=\sum\limits_{n=0}^\infty a_n(z-t_0)^n,$$
$z\in D(t_0,\delta)$,
which is well defined and holomorphic in $D(t_0,\delta)$.  We have that 
$$\Gamma^{-1}: \Gamma(D(t_0,\varepsilon))\rightarrow \mathbb{C}$$ is a holomorphic function. We consider the function $$F=go\Gamma^{-1}:
\Gamma(D(t_0,\varepsilon))\rightarrow \mathbb{C},$$ (where $\Gamma(D(t_0,\varepsilon))$ is an open set) which is a holomorphic function and $$(Fo\Gamma)(t)=f(t),t\in D(t_0,\varepsilon)
\cap I.$$ Therefore, there exist $b_n\in \mathbb{C}$,$n=1,2,3,...$ and $ \delta>0 $ such that 
$$ F(z)=\sum\limits_{n=0}^\infty b_n(z-\gamma(t_0))^n $$\\
for every  $z \in D(\gamma(t_0),\delta)\subseteq\Gamma(D(t_0,\varepsilon))$ and so $$ f(t)=(Fo\gamma)(t)=\sum\limits_{n=0}^\infty b_n(\gamma(t)-\gamma(t_0))^n $$
in an interval $ (t_0- s,t_0+s)$ where $ s>0 $.
\\

$(ii)\Rightarrow(i)$
We cosinder the function $$G(z)=\sum\limits_{n=0}^\infty b_n(z-\gamma(t_0))^n,$$ 
$z\in D(\gamma(t_0),s)$. We choose $ a>0 $ with $ a<\varepsilon $ such that $ \Gamma(D(t_0,a))\subseteq D(\gamma(t_0),s)$.
The function 
$$Go\Gamma :D(t_0,a)\rightarrow \mathbb{C}$$ is holomorphic. 
Therefore, there exist
$a_n\in \mathbb{C}$,$n=1,2,3,...$ such that
$$(Go\Gamma)(z)=\sum\limits_{n=0}^\infty a_n(z-t_0)^n,
z\in D(t_0,a)$$ and consequently  
$$f(t)=G(\gamma(t))=\sum\limits_{n=0}^\infty a_n(t-t_0)^n,t\in(t_0-
a,t_0+a).$$
\end{proof}

As we have seen before, if $ \delta:[0,2\pi]\rightarrow\mathbb {C} $ is an analytic Jordan curve, $ \Omega $ is the interior of  $ \delta $ and  $ \varphi:D\rightarrow \Omega $ is a Riemann function, then $ \varphi $ has a conformal extension in an open neighbourhood of $ \overline{D} $, $ W $,   and $ \sup\limits_{|z|\leq1} |\varphi^{(l)}(z)|< +\infty$, $ \sup\limits_{|z|\leq1} |(\varphi^{-1})^{(l)}(z)|< +\infty$ for every $ l=0,1,2 \dots $. Let $ \gamma:[0,2\pi]\rightarrow \mathbb{C} $ be a Jordan analytic curve with $ \gamma(t)=\varphi(e^{it}) $ for $ t \in [0, 2\pi] $. We extend $ \gamma $ to a $ 2\pi $-periodic function. At the following, when we use the symbols $ \Omega, \gamma, \varphi, W $ , they will have the same meaning as mentioned above, unless something else is denoted.

\begin{defi}
A function $ f:\gamma^{*}\rightarrow \mathbb{C} $ is real analytic at $ \gamma(t_0) $, if there exists a power series $ \sum\limits_{n=0}^\infty a_n(t-t_0)^n $ with a radius of convergence $ \delta>0 $ such that $ f(\gamma(t))=\sum\limits_{n=0}^\infty a_n(t-t_0)^n $  for every $ t \in(t_0-\delta,t_0+\delta) $.
\end{defi}

\begin{defi}
We will say that a function $ f:\gamma^{*}\rightarrow \mathbb{C} $ is extendable at $ \gamma(t_0) $ if there exists  an open set $ V $ with $ \gamma(t_0)\in V $ and a holomorphic function $ F:V\rightarrow \mathbb {C} $ such that $ F|_{V\cap \gamma^{*}}=f|_{V\cap \gamma^{*}} $. If a function $ f \in C^{\infty}(\gamma^{*}) $ is not extendable at any $ \gamma(t_0) $, then it belongs to $ U_5(\gamma^{*}, \infty) $ .
 
\end{defi}

The next proposition is , in fact, equivalent to the lemma 5.7.

\begin{prop}
A function $ f:\gamma^{*}\rightarrow \mathbb{C} $ is real analytic at $ \gamma(t_0) $ if and only if $ f $ is extendable at $ \gamma(t_0) $.
\end{prop}
\begin{proof} 
$ \Rightarrow: $ If $ f $ is real analytic at $ \gamma(t_0) $, then from lemma 5.7 
$$f(\gamma(t))=
 \sum\limits_{n=0}^\infty b_n(\gamma(t)-\gamma(t_0))^n$$ for every $t\in(t_0-   
 \epsilon,t_0+
 \epsilon)$, for some $ b_n \in \mathbb {C}, \epsilon>0 $ and therefore, 
 $$F(z)=
 \sum\limits_{n=0}^\infty b_n(z-\gamma(t_0))^n$$ for $t\in D(z_0,\epsilon)$ is an extension of $ f $ at $ \gamma(t_0) $\\
 $ \Leftarrow: $ If $ f $ is extendable at $ \gamma(t_0) $, where $ f $ is the extension, then  
 $$f(\gamma(t))=F(\gamma(t))=\sum\limits_{n=0}^\infty b_n(\gamma(t)-\gamma(t_0))^n$$ 
 for every $z\in(t_0-\epsilon,t_0+\epsilon)$, for some $ b_n \in \mathbb {C}, \epsilon>0 $ and as a result, again from lemma 5.7 we conclude that $ f $ is real analytic at $ \gamma(t_0) $.
\end{proof}

The previous Proposition and the theorem 4.23, immediately prove the following theorem

\begin{theorem}
For $ k=0,1,2, \dots, \infty $ the set of functions $ f \in C^{k}(\gamma^{*}) $ which are nowhere real analytic is dense $ G_{\delta} $ subset of $ C^{k}(\gamma^{*}) $.
\end{theorem}

The following more general result was suggested by J.-P. Kahane.
\\

Let $ z_n, n=0,1,2, \cdots $ be a dense sequence of $ T $ and let $ M=(M_n), n=0,1,2, \cdots $ be a sequence of real numbers. For $ k,l \in \mathbb{N} $ we denote the set
$$ F (M,k,z_l ) = \left\{ f \in C^{\infty}(T) : \left| \frac{d^{n}(f)}{d\vartheta^{n}} (z_l) \right| \leq M_n k^n  \forall n=0,1,2, \dots  \right\} $$

\begin{lemma}
$ F (M,k,z_l ) $ is closed and has an empty interior in $C^{\infty}(T)$.
\end{lemma}
\begin{proof}
It is obvious that this set is closed. Let suppose that there is $ f\in \left(F (M,k,z_l )\right)^{o}$. Then, there exist $ m\in \mathbb{N} $ and  $ \varepsilon>0 $ such that 
$$V=\left\{ u \in C^{\infty}(T) : \sup\limits_{\vartheta\in \mathbb{R}}\left| \frac{d^{n}(u)}{d\vartheta^{n}}(e^{i\vartheta})-\frac{d^{n}(f)}{d\vartheta^{n}} (e^{i\vartheta}) \right|<\varepsilon \right\}\subseteq F (M,k,z_l ) $$ for all $ n=0,1,2, \cdots , m$.
\\
Let $$ A>k^{m+1} M_{m+1}+ \sup\limits_{\vartheta\in \mathbb{R}} \left| \frac{d^{m+1}(f)}{d\vartheta^{m+1}}(e^{i\vartheta}) \right|$$
with $ A>0 $ and $ b \in \mathbb{N} $ with $ \frac{A}{b}<\varepsilon $. Then, let $ a $ be the trigonometric polynomial $ a(e^{i\vartheta})=\frac{A}{b^{m+1}}e^{ib\vartheta} $. We have that 
$$ \sup\limits_{\vartheta\in \mathbb{R}} \left| \frac{d^{n}(a)}{d\vartheta^{n}}(e^{i\vartheta}) \right|<\varepsilon $$ for $ n=0,1,2, \cdots, m $ and
$$ \left| \frac{d^{n}(a)}{d\vartheta^{m+1}}(e^{i\vartheta}) \right|=A $$
Therefore, $ a+f \in V\subseteq F (M,k,z_l ) $ and so we have that $$A - \left| \frac{d^{m+1}(a)}{d\vartheta^{m+1}} (z_l) \right| \leq M_{m+1} k^{m+1} $$
which is a contradiction.
\\

As result of this Lemma we have from Baire's theorem that the set 
$ \bigcap\limits_{l=1}^{\infty}\bigcap\limits_{k=1}^{\infty}(C^{\infty}(T)\backslash F (M,k,z_l ) )$ is a dense and $ G_\delta$ subset of $C^{\infty}(T)$. For $ M_n=n! $ if $ f\in\bigcap\limits_{l=1}^{\infty}\bigcap\limits_{k=1}^{\infty}(C^{\infty}(T)\backslash F (M,k,z_l ) ) $
then we easily can see that $f$ is nowhere real analytic at $ T $ and therefore according to Proposition 5.10 we have that $ U_5(T) $ contains a dense and $ G_\delta $ subset of $C^{\infty}(T)$.
\end{proof}

\begin{defi}
Let $ f:\Omega\rightarrow\mathbb{C}$ be a holomorphic function. We say that $ f $ belongs to the set $ A^{\infty}(\Omega) $ if for every $ l=0,1,2,3 \dots $, $ f^{(l)} $ can be extended to a continuous function in $ \overline{\Omega} $.  The topology of $A^{\infty}(\Omega)$ is defined by the seminorms  
$$ \sup\limits_{z\in\overline{\Omega}} |f^{(l)}(z)| , l=0,1,2,3\dots $$ 
It can easily be proven that for a function $ g=fo\varphi $ we have that $ f\in A^{\infty}(\Omega) $ if and only if $ g \in A^{\infty}(D) $ and that the function $ h:A^{\infty}(\Omega)\rightarrow A^{\infty}(D) $ with $ h(f)=fo\varphi=g $ is a linear one and a topological homeomorphism.
\end{defi}
\begin{defi}
We say that an harmonic function $ f:\Omega\rightarrow \mathbb {C} $ belongs to the set $ H_{\infty}(\Omega) $ if $ L_{k,l}(f) $ can be extended to a continuous function in $ \overline{\Omega} $ for every $ k,l=0,1,2 \dots $, where $ L_{k,l}=\frac{\partial^k}{\partial z^k} \frac{\partial^l}{\partial \overline{z}^l}$. We observe that if $ k\geq 1 $ and $ l\geq 1 $ then $ L_{k,l}(f)\equiv0 $, because $ f  $ is harmonic function. The topology of $ H_{\infty}(\Omega) $ is defined by the seminorms 
$$ \sup\limits_{z\in\overline{\Omega}} \left| \frac{\partial^k f(z)}{\partial z^k} \right|, \sup\limits_{z\in\overline{\Omega}} \left|\frac{\partial^l f(z)}{\partial \overline{z}^l}\right|, k,l\geq 0.$$
As in the previous definition, we can see that for a function $ g=fo\varphi $ we have that $ f \in H_{\infty}(\Omega) $ if and only if $ g \in H_{\infty}(D) $ and that the function $ h:H_{\infty}(\Omega)\rightarrow H_{\infty}(D) $ with $ h(f)=fo\varphi=g $ is isomorphism. 
\end{defi}

At this point, we use the conformal function $ \varphi $ in order to extend a previous result. Specifically, we easily can conclude that $ C^{\infty}(\partial \Omega)=A^{\infty}(\Omega)\oplus A^{\infty}_0(\Omega)=H_{\infty}(\Omega) $, where $A^{\infty}(\Omega)=\{f\in A^{\infty}(\Omega): f(\varphi(0))=0\}$.

Now, we will consider the case of a domain $ \Omega $ which is bounded by a finite number of pairwise disjoint Jordan analytic curves $ \gamma_1,\gamma_2,\dots ,\gamma_m $. For a function $ u: \partial \Omega\rightarrow \mathbb {C} $ we denote that $ u\in C^{\infty}(\partial \Omega)  $ if $ u|_{\gamma^{*}_{j}} \in C^{\infty}(\gamma^{*}_{j}) $ for every $ j=1,2,3 \dots, m $. Let $ \phi_j:D\rightarrow V_j $ a Riemann function, where $V_j$ is the interior of $ \gamma_j $ for every $ j=1,2,3 \dots, m $. The topology of $C^{\infty}(\partial \Omega)$ is defined by the seminorms 
$$ \max\limits_{j=1,2, \dots , m}\sup\limits_{\vartheta\in \mathbb{R}} \left| \frac{d^l (u(\gamma_jo\varphi_j))}{d\vartheta^l}(e^{i\vartheta})\right|, l=0,1,2, \dots$$
At the following the symbols $ \Omega, \gamma_j, \phi_j, V_j$ will have the same notion as above, unless something else is denoted.
\\
\begin{defi}
We say that a holomorphic function belongs to the class $ A^{\infty}(\Omega) $ if for every $ l=0,1,2 \dots $, $ f^{(l)} $ is extended to a continuous function   in $\overline{\Omega}  $. The topology of $ A^{\infty}(\Omega) $ is defined by the seminorms $ \sup\limits_{z\in\Omega} \left| f^{(l)}(z)\right|, l=0,1,2,3 \dots$
\end{defi}

According to a well-known result ([3]) every holomorphic function $ f $ in $ \Omega $ can be uniquely written as $ f=f_1+f_2+\dots+f_{m-1} $ where $ f_j $ is holomorphic in $ W_j^{c} $ for $ j=0,1,2, \dots, m-1 $ and $ \lim\limits_{z\rightarrow\infty}f_k(z)=0 $ for $ k=1,2, \dots, m-1 $, where $ W_j, j=0,1,2,\dots,m-1 $ are the the connected components of $ \left(\overline{\Omega}\right)^c $ and $ W_0 $ is the non-bounded connected component of $ \left(\overline{\Omega}\right)^c $. It is obvious that if $ f\in A^{\infty}(\Omega) $ then $ f_j\in A^{\infty}(W_j^c) $ and as a result we have that $ A^{\infty}(\Omega)=A^{\infty}(W_0^c)\oplus A_0^{\infty}(W_1^c)\oplus\dots\oplus A_0^{\infty}(W_{m-1}^c) $, where $A_0^{\infty}(W_k^c)=\{f\in A^{\infty}(W_k^c):\lim\limits_{z\rightarrow\infty}f(z)=0\}$ for $ k=1,2,\dots,m-1 $.
\\

At this point, we want to examine if it is true that $ C^{\infty}(\partial\Omega)=A^{\infty}(\Omega)_{|\partial\Omega}\oplus \overline{X} $ for a subspace $ X $ of $ A^{\infty}(\Omega)_{|\partial\Omega} $ in order to extend the corresponding result for $ \Omega=D $. Unfortunately, we can prove that this is false, when $ m\geq 2 $, in other words, when $ \Omega $ is not simply connected. 
\\
Indeed, let $ z_0 $ be a point in a bounded connected component of $ \left(\overline{\Omega}\right)^c $. Also, let $ u:\mathbb {C}\setminus \{z_0\}\rightarrow \mathbb {R} $  be a function with $ u(z)=ln|z-z_0| $ for $ z\in \mathbb {C}\setminus \{z_0\} $. This function is an harmonic one and  $ u|_{\partial\Omega} \in C^{\infty}(\partial\Omega) $. If there exist $ f,g \in A^{\infty}(\Omega) $ such that $ u|_{\partial\Omega}=f|_{\partial\Omega}+\overline{g}|_{\partial\Omega} $, then we have that 
$$ u|_{\partial\Omega}=Re(u|_{\partial\Omega})=Re\left(f|_{\partial\Omega}+\overline{g}|_{\partial\Omega} \right)=Re (f)|_{\partial\Omega}+Re(g) |_{\partial\Omega}=Re(f+g)|_{\partial\Omega}$$
So,we have that $ u, Re(f+g)$ are harmonic functions in $ \Omega $, continuous functions in $ \overline{\Omega} $ and $ u|_{\partial\Omega}= Re(f+g)|_{\partial\Omega} $. Therefore, from a known theorem we have that $ u=Re(f+g) $ in $ \Omega $. The function $ \frac{z-z_0}{e^{(f+g)(z)}} $ is holomorphic in $ \Omega $ and $$ \left|\frac{z-z_0}{e^{(f+g)(z)}}\right|=\frac{|z-z_0|}{e^{\left(Re(f+g)\right)(z)}}=\frac{|z-z_0|}{e^{u(z)}}=\frac{|z-z_0|}{|z-z_0|}=1 $$ for every $ z\in \Omega $ and as a result, because $\Omega$ is a domain, we have that $ \frac{z-z_0}{e^{(f+g)(z)}} $ is constant in $ \Omega $ and so there exists $ c\in\mathbb {C} $ with $ |c|=1 $ and $ ce^{(f+g)(z)}=z-z_0 $ for every $ z\in \Omega $. Therefore, there exists a holomorphic logarithm of $ z-z_0 $ in $ \Omega $, while there is a Jordan curve in $ \Omega $ which includes $ z_0 $ in its interior, which it is well-known that this is a contradiction.
\\

\begin{defi}
We will say that a function  $ f\in A^{\infty}(\Omega) $ is extendable at $ \gamma_j(t_0) $ for some $ j\in \{1,2 \dots, m\} $ if there exists $ \delta>0 $ and a holomorphic function $ F:D(\gamma_j(t_0),\delta)\rightarrow \mathbb {C} $ such that $ F|_{D(\gamma(t_0),\delta)\cap \overline{\Omega}}=f|_{D(\gamma(t_0),\delta)\cap \overline{\Omega}} $.
\end{defi}

At this point, we will need the following lemma:

\begin{prop}
A function  $ f\in A^{\infty}(\Omega) $ is extendable at $ \gamma_j(t_0) $ if and only if $ f|{\gamma^{*}_j} $ is real analytic at $ \gamma_j(t_0) $.
\end{prop}

\begin{proof}
If $ f $ is extendable at $ \gamma_j(t_0) $ then $ f|_{\gamma^{*}_j} $ is real analytic at $ \gamma_j(t_0) $ from Proposition 5.10.\\
Conversely, if $ f|_{\gamma^{*}_j} $ is real analytic at $ \gamma_j(t_0) $ , then from Proposition 5.10 there exists $ \delta>0 $ and a holomorphic function $ g:D(\gamma_j(t_0),\delta)\rightarrow \mathbb {C} $ such that $ f|_{\gamma^{*}_j \cap D(\gamma_j(t_0),\delta)}= g|_{\gamma^{*}_j \cap D(\gamma_j(t_0),\delta)}$. Let $ F:D(\gamma_j(t_0),\delta)\rightarrow \mathbb {C} $ be a function with
$$ F(z) = \left\{
\begin{array}{ c l }
f(z),   &    z\in D(\gamma_j(t_0),\delta) \cap \overline{\Omega}  \\
g(z),   &    z\in D(\gamma_j(t_0),\delta) \cap (\mathbb {C}\setminus \Omega)
\end{array}
\right. $$
The function is well-defined  since $ f|{\gamma^{*}_j \cap D(\gamma_j(t_0),\delta)}= g|{\gamma^{*}_j \cap D(\gamma_j(t_0),\delta)}$, is holomorphic in $D(\gamma_j(t_0),\delta)\setminus \gamma^{*}_j$ and continuous in $D(\gamma_j(t_0),\delta)$. So, according to Proposition 2.1, we have that $ f $ is holomorphic in $D(\gamma_j(t_0),\delta)$ and so $ F $ is an extension at $ \gamma_j(t_0) $.
\end{proof}

With the previous Proposition, we can prove the following theorem.

\begin{theorem}
The set of functions $ f\in A^{\infty}( \Omega )$ such that for every $ j\in \{1,2 \dots, m\} $, $f|{\gamma^{*}_j} $ is nowhere real analytic, is a dense and $ G_{\delta} $ subset of $ A^{\infty} (\Omega) $.
\end{theorem}

\begin{proof}
Because of the Proposition 5.17, the set of functions $ f\in A^{\infty} (\Omega) $ such that for every $ j\in \{1,2 \dots, m\}, f|{\gamma^{*}_j} $ is nowhere real analytic, is the same with the set of functions $ f\in A^{\infty}( \Omega )$ such that $ f $ is nowhere extendable in $ \partial\Omega $. However, the last set is a dense and $ G_{\delta} $ subset of $ A^{\infty}( \Omega) $ according to [8] and [10]. Therefore, the same is true for the first set.
\end{proof}

\begin{remark}
According to [6] if $ \gamma $ is a analytic Jordan  curve, then the same is true for the parametrization of $ \gamma $ by the arc length. Therefore, the Theorem 5.18 is also valid for this parametrization of $ \gamma $.
\end{remark}

\section{The complex method and non-extendability}
At this paragraph we will use a complex method to prove that the results of Paragraph 3. 
\\
Firstly, we will consider the segment $ [0,1] $ and we will prove that for every $ k=1,2, \dots $ or  $k=\infty $ the set of functions of $C^{k}[0,1]$ which are not extended to a holomorphic function at a disk with center at any point of $ [0,1] $ is a dense $ G_\delta $ subset of the respective $ C^{k}[0,1] $. The same is true for the continuous functions or for semidisks instead of disks .\\
For the proof, we will need the following lemma. 

\begin{lemma}
Let $ D(z_0,r)$ be a disk with center $ t_0\in (0,1) $ and radius $ r>0 $ and let $ +\infty > M > 0 $.  The set of continuous functions $ f:[0,1]\rightarrow \mathbb{C} $, for which there exists a holomorphic  function $ F  $ at $ D(t_0,r)$ where $ F $ is bounded by $ M $ and $ F|_{D(t_0,r)\cap [0,1]}=f|_{D(t_0,r)\cap [0,1]} $,is a closed subset of $ C\left([0,1]\right) $ and it has empty interior.
\end{lemma}

\begin{proof}
Let $ A $ be the set of continuous function$ f:[0,1]\rightarrow \mathbb{C} $ which are extended to a holomorphic  function $ F  $ at $ D(t_0,r)$ where $ F $ is bounded by $ M $and $ F|_{D(t_0,r)\cap [0,1]}=f|_{D(t_0,r)\cap [0,1]} $. If this set has not empty interior, then there is a function $ f $ in the interior of A and $ \delta>0  $ such that $$\left\{g \in C\left([0,1]\right): \sup \limits_{t \in [0,1]}\left| f(t)-g(t) \right| < \delta \right\}\subseteq A$$ We choose $ z_0 \in D(t_0,r)\backslash [0,1] $ and $ a>0 $ with $ a<\delta \inf \limits_{t\in [0,1]}|t-z_0| $. The function $ h(t)=f(t)+ \dfrac{a}{2(t-z_0)} $ for $ t\in [0,1] $ belongs to $ A $ and therefore has a holomorphic and bounded extension $ G $ at $ D(t_0,r)$ with $ G|_{D(t_0,r)\cap [0,1]}=h|_{D(t_0,r)\cap [0,1]} $. However, $ G(z)=F(z)+ \dfrac{a}{2(z-z_0)} $ for $ z\in D(t_0,r)\backslash \{z_0\}$ by analytic continuity since they are equal at $ [0,1] $. As a result $ G $  is not bounded at $ D(t_0,r)$ which is a contradiction. So, $  A $ has empty interior.\\
Let $ (f_n)_{n\geq 1} $ be a sequence in $ A$ where $ f_n $ uniformly converges at $ [0,1] $ to a function $ f $. Then,  for $ n=1,2, \dots $ there are holomorphic and bounded functions by $ M $ $ F_n:D(t_0,r)\rightarrow\mathbb{C} $ with $ F_n|_{D(t_0,r)\cap [0,1]}=f_n|_{D(t_0,r)\cap [0,1]} $. By Montel's theorem, there is a subsequence of $ (F_n) $, $ (F_{k_n}) $ which converges uniformly at the compact subsets of $ D(t_0,r) $ to a function $ F $ which is holomorphic and bounded by $ M $at $ D(t_0,r) $. Because $ F_{k_n} \rightarrow f$ at $ D(t_0,r)\cap [0,1] $ we have that $ F|_{D(t_0,r)\cap [0,1]}=f|_{D(t_0,r)\cap [0,1]} $ and so $ f \in A $. Therefore, A is a closed set.
\end{proof}

The above lemma immediately proves the following theorem.

\begin{theorem}
The class $ U_5([0,1],0) $ is a dense $ G_\delta $ subset of $ C[0,1] $.
\end{theorem}

\begin{proof}
We denote $ A(M,t_0,r) $  the set $ A $ of the above lemma. Then, if we consider a dense sequence $ z_l $ of $ (0,1) $ , then the set $ \bigcap \limits^{\infty}_{l=1}\bigcap \limits^{\infty}_{n=1}\bigcap \limits^{\infty}_{M=1} \left( C[0,1] \backslash A(M,t_0,r) \right) $ is a dense $ G_\delta $ subset of $ C[0,1] $ due to Baire's Theorem and coincides with $ U_5([0,1],0) $.
\end{proof}

\begin{remark}
From the proof of the theorem we can see that we have the same result  for the classes  $ U_5([0,1],k) $ for $ k=1,2 \dots, \infty $.
\end{remark}

\begin{theorem}
Let $ L $ be a closed set without isolated points. Then the class of continouos functions at $ L $ for which there is no pair of disk $D(z_0,r),z_0\in L,r>0$ and holomorphic function
$f: D(z_0,r)\rightarrow \mathbb{C}$ such that 
$f|_{D(z_0,r)\cap L}=u|_{D(z_0,r)\cap L}$ is dense $ G_\delta $ subset of $ C(L) $.
\end{theorem}

\begin{proof}
The proof is the same with theorem's 6.2 proof.
\end{proof}

Now we will prove the result of the first theorem of this paragraph for the case of a semi-disk. Now, we need the following lemma.

\begin{lemma}
Let $ D(t_0,r)\cap \{z \in \mathbb{C}:Im(z)>0\}$ be a semi-disk with $ t_0\in (0,1) $ and radius $ 0<r\leq min(t_0,1-t_0) $ and let $ M>0 $.  The set of continuous functions $ f:[0,1]\rightarrow \mathbb{C} $, for which there exists a continuous function $ F  $ at $ \left(D(t_0,r)\cap \{z \in \mathbb{C}:Im(z)>0\}\right) \cup [t_0-r,t_0-r] $ where $ F $ is bounded by $ M $, $ F $ is holomorphic at $ D(t_0,r)\cap \{z \in \mathbb{C}:Im(z)>0\}$ and $ F|_{D(t_0,r)\cap [0,1]}=f|_{D(t_0,r)\cap [0,1]} $,is a closed subset of $ C\left([0,1]\right) $ and it has empty interior.
\end{lemma}

\begin{proof}
Let $ B $ be the set of continuous functions $ f:[0,1]\rightarrow \mathbb{C} $ for which there exists a continuous function $ F  $ at $ \left(D(t_0,r)\cap \{z \in \mathbb{C}:Im(z)>0\}\right) \cup [t_0-r,t_0-r] $ where $ F $ is bounded by $ M $, $ F $ is holomorphic at $ D(t_0,r)\cap \{z \in \mathbb{C}:Im(z)>0\}$ and $ F|_{D(t_0,r)\cap [0,1]}=f|_{D(t_0,r)\cap [0,1]} $. If this set has not empty interior, then there is a function $ f $ in the interior of $  B$ and $ \delta>0  $ such that $$\left\{g \in C\left([0,1]\right): \sup \limits_{t \in [0,1]}\left| f(t)-g(t) \right| < \delta \right\}\subseteq B$$ We choose $ z_0 \in D(t_0,r)\backslash [0,1] $ with $ Im(z_0)>0 $ and $ a>0 $ with $ a<\delta \inf \limits_{t\in [0,1]}|t-z_0| $. The function $ h(t)=f(t)+ \dfrac{a}{2(t-z_0)} $ for $ t\in [0,1] $ belongs to $ B $ and therefore has a continuous and bounded extension $ G $ at $ \left(D(t_0,r)\cap \{z \in \mathbb{C}:Im(z)>0\}\right) \cup [t_0-r,t_0-r] $ with $ G|_{D(t_0,r)\cap [0,1]}=h|_{D(t_0,r)\cap [0,1]} $ which is holomorphic at $ D(t_0,r)\cap \{z \in \mathbb{C}:Im(z)>0\}$. We easily can see that $ G(z)=F(z)+ \dfrac{a}{2(z-z_0)} $ for $ z\in D(t_0,r)\backslash \{z_0\}$ . Indeed, by Schwarz Reflection Principle there exists a holomorphic function $ H:D(t_0,r)\backslash \{z_0, -z_0\}$ with $ H(z)=G(z)-F(z)- \dfrac{a}{2(z-z_0)} $ at $ \left(D(t_0,r)\cap \{z \in \mathbb{C}:Im(z)>0\}\right) \cup [t_0-r,t_0-r] $. Therefore,by analytic continuity, because $ H=0 $ at $ D(t_0,r)\cap [0,1] $, $ G(z)-F(z)- \dfrac{a}{2(z-z_0)} =0$ at $ D(t_0,r)\cap \{z \in \mathbb{C}:Im(z)>0\}$. As a result $ G $ is not bounded at $ D(t_0,r)$ which is a contradiction. So, $  B $ has empty interior.\\
Now, we will prove that $ B $ is closed. Let $ (f_n)_{n\geq 1} $ be a sequence in $ B$ where $ f_n $ uniformly converges at $ [0,1] $ to a function $ f $. Then,  for $ n=1,2, \dots $ there exist continuous functions $ F_n  $ at $ \left(D(t_0,r)\cap \{z \in \mathbb{C}:Im(z)>0\}\right) \cup [t_0-r,t_0-r] $ where $ F_n $ are bounded by $ M $ and $ F_n $ are holomorphic at $ D(t_0,r)\cap \{z \in \mathbb{C}:Im(z)>0\}$ with $ F_n|_{D(t_0,r)\cap [0,1]}=f_n|_{D(t_0,r)\cap [0,1]} $. There exists a bijective and holomorphic function $ \Psi :I\cup D \rightarrow \left(D(t_0,r)\cap \{z \in \mathbb{C}:Im(z)>0\}\right) \cup [t_0-r,t_0-r] $, where $ I= \left\{ e^{it}: 0\leq t\leq 1 \right\} $. Then, if $ G_n=F_no\Psi $ ,$ g_n=f_no\Psi $ and $ g=fo\Psi $  for $ n=0,1,2, \dots $, we have that $ g_n $ uniformly converges at $ I $ to $ g $. Also, $ g_n, g, G_n $ are continuous functions and $ G_n $ are holomorphic at $ D $ and bounded by $ M $. By Montel's theorem, there is a subsequence of $ (G_n) $, $ (G_{k_n}) $ which converges uniformly at the compact subsets of $ D $ to a function $ G $ which is holomorphic and bounded by $ M $ at $ D $. We can suppose without loss of generality that $ (G_n)=(G_{k_n}) $, because otherwise we can consider the sequence $ (W_n)=(G_{k_n}) $ and follow the same proof for this function. 
Now, it is sufficient to prove that for any circular sector which has boundary $ [0,e^{ia}]\cup [0,e^{ib}] \cup \left\{ e^{it}: a\leq t\leq b \right\}$ with $ 0<a<b<1 $,  $ (G_n) $ converges uniformly at this sector, because then we will have that the limit of $ (G_n) $ , which is $ g $ at the arc and $ G $ at the other part of the circular sector, will be a continuous function. So, let $ K $ be a closed circular sector which has boundary $ [0,e^{ia}]\cup [0,e^{ib}] \cup \left\{ e^{it}: a\leq t\leq b \right\}$ with $ a\leq t\leq b $. We will prove that $ (G_n) $ is uniformly Cauchy at $ K $. 
From [6], we know that for every $ n $, the radial  limits of $ G_n $ exist almost everywhere and so we can consider the respective functions $ g_n $ on the whole circle which are extensions of the previous $ g_n $. These $ g_n $ are also bounded by $ M $.\\
 Let $ \varepsilon>0 $ be a positive number. For the Poisson kernel $ P_r$ and for every $ n=0,1,2 \dots $ we have that $G_n(re^{i\theta})=\dfrac{1}{2\pi}\int\limits_{-\pi}^{\pi}P_r(t)g_n(\theta-t)dt $.  We choose $ 0<\delta < min\{1-b,a\} $. There exists $ 0<r_0<1 $ such that $ \sup\limits_{\delta\leq |t|\leq \pi} P_r(t) < \frac{\varepsilon}{8M} $ for every $ r\in [r_0,1) $. Then, at $ K\cap \{z \in \mathbb {C}: |z|\leq r_0\} , (G_n)  $ is uniformly Cauchy and thus there exists $ n_1 $ such that for every $ n,m\geq n_1 $, 
 $$ \sup\limits_{z \in K\cap \{z \in \mathbb {C}: |z|\leq r_0\} }|G_n(z)-G_m(z)|<\dfrac{\varepsilon}{2}$$ 
 In addition ,because $ g_n $ converges uniformly to $ g $ at $ I $ there exists $ n_2 $ such that for every $ n,m\geq n_2 $, $ \sup\limits_{z \in I} |g_n(z)-g_m(z)|<\dfrac{\varepsilon}{4}$.
 So, for $ n,m\geq \max\{n_1,n_2\} $ , for $ \theta \in [a,b] $ and for $ 1>r>r_0 $
 $$ |G_n(re^{i\theta})-G_m(re^{i\theta})|\leq  \dfrac{1}{2\pi}\int\limits_{-\pi}^{\pi}P_r(t)|g_n(\theta-t)-g_m(\theta-t)|dt=$$
 $$\dfrac{1}{2\pi}\int\limits_{-\delta}^{\delta}P_r(t)|g_n(\theta-t)-g_m(\theta-t)|dt+\dfrac{1}{2\pi}\int\limits_{\delta\leq |t|\leq \pi}P_r(t)|g_n(\theta-t)-g_m(\theta-t)|dt\leq $$
 $$\dfrac{1}{2\pi}\int\limits_{-\delta}^{\delta}P_r(t)\sup\limits_{z \in I} |g_n(z)-g_m(z)|dt+\dfrac{\sup\limits_{\delta\leq |t|\leq \pi} P_r(t)}{2\pi}\int\limits_{\delta\leq |t|\leq \pi}2Mdt\leq $$
 $$\dfrac{\varepsilon}{4}\dfrac{1}{2\pi}\int\limits_{-\delta}^{\delta}P_r(t)dt+\dfrac{\varepsilon}{16M\pi}\int\limits_{\delta\leq |t|\leq \pi}2Mdt\leq $$
 $$\dfrac{\varepsilon}{4}\dfrac{1}{2\pi}\int\limits_{-\pi}^{\pi}P_r(t)dt+\dfrac{\varepsilon}{4}=
  \dfrac{\varepsilon}{2} $$
  The inequality also holds for $ r=1 $. Therefore $ (G_n) $ is uniformly Cauchy and thus $ B $ is a closed subset of $ C([0,1]) $
\end{proof}

\begin{theorem}
The class $ U_2([0,1],0) $ is a dense $ G_\delta $ subset of $ C[0,1] $.
\end{theorem}

\begin{proof}
We denote $ B(M,t_0,r) $  the set $ B $ of the above lemma. Then, if we consider a dense sequence $ z_l $ of $ (0,1) $ , then the set $ \bigcap \limits^{\infty}_{l=1}\bigcap \limits^{\infty}_{n=1}\bigcap \limits^{\infty}_{M=1} \left( C[0,1] \backslash B(M,t_0,r) \right) $ is a dense $ G_\delta $ subset of $ C[0,1] $ due to Baire's Theorem and coincides with $ U_2([0,1],0) $.
\end{proof}

\begin{remark}
From the proof of the theorem we can see that we have the same result  for the classes  $ U_2([0,1],k) $ for $ k=1,2 \dots, \infty $. Furthermore, the same is true for the classes $ U_3([0,1],k) $,$ U_4([0,1],k) $,$ U_5([0,1],k) $ for $ k=0,1,2 \dots, \infty $
\end{remark}

\noindent \textbf{Acknowledgement}: We would like to thank P. Gauthier, 
J.-P. Kahane, G. Koumoullis, M. Maestre, S. Mercourakis, M. Papadimitrakis and 
A. Siskakis for helpful communications.\\

\begin{center}
\textbf{REFERENCES}
\end{center}

\noindent
[1]: L. Ahlfors, Complex Analysis, An Introduction to the Theory of Analytic 
Functions of One Complex Variable, Third Edition, McGraw-Hill, Book Company

\noindent
[2]: F. S. CATER, Differentiable, Nowhere Analytic functions, Amer. Math. Monthly 
91 (1984) no. 10, 618-624.

\noindent
[3]: G. Costakis, V. Nestoridis and I. Papadoperakis, Universal Laurent series,
Proc. Edinb. Math. Soc. (2) 48 (2005), no. 3, 571–583

\noindent
[4]: A. Daghighi and S. Krantz, A note on a conjecture concerning boundary 
uniqueness, arxiv 1407.1763v2 [math.CV], 12 Aug. 2015

\noindent
[5]: J. Garnett, Analytic Capacity and Measure, Lecture Notes in Mathematics, vol. 
277, Springer-Verlag Berlin, Heidelberg, New York, 1972

\noindent
[6]: K. Hoffman, Banach Spaces of Analytic Functions, 1962 Prentice Hall Inc.,
Englewood Cliffs, N. J.

\noindent
[7]: S. Krantz and H. Parks, A primer of Real Analytic Functions, Second edition,
2002 Birkhauser Boston

\noindent
[8]: V. Nestoridis, Non extendable holomorphic functions, Math. Proc. Cambridge
Philos. Soc. 139 (2005), no. 2, 351–360.

\noindent
[9]: V. Nestoridis and A. Papadopoulos, Arc length as a conformal parameter for 
locally analytic curves, arxiv: 1508.07694

\noindent
[10]: Nestoridis, V., Zadik, I., Pade approximants, density of rational functions in
$A^\infty(\Omega)$ and smoothness of the integration operator. (arxiv: 1212.4394.) J.
M. A. A. 423 (2015), no.2, 1514-1539

\noindent
[11]: W. Rudin, Function theory in polydiscs, W. A. Benjamin 1969
\\

\noindent
University of Athens\\
Department of Mathematics\\
157 84 Panepistemiopolis\\
Athens\\
Greece\\

\noindent
E-mail addresses:\\
lefterisbolkas@gmail.com (E. Bolkas)\\
vnestor@math.uoa.gr (V. Nestoridis)\\
chris$\_$panagiwtis@hotmail.gr (C. Panagiotis)

\end{document}